\documentclass[12pt]{amsart}

\usepackage{amsfonts, amssymb, curves, epic, epsf, amsmath, amsthm}
\usepackage{amsfonts, amssymb, amsmath, amsthm, graphicx}
\usepackage{epstopdf}
\newcommand{\R}{\mathbb{R}}
\newcommand{\N}{\mathbb{N}}
\newcommand{\Q}{\mathbb{Q}}
\newcommand{\Z}{\mathbb{Z}}
\newcommand{\T}{\mathbb{T}}

\newcommand{\Lim}{\varprojlim} %% Inverse limit was {\underleftarrow\lim\ }
\newcommand{\OP}{\Omega_{\Phi}}

\newcommand{\larr}{\left( \begin{array}{c}}
\newcommand{\rarr}{\end{array} \right) }

\newcommand{\lsqarr}{\left[ \begin{array}{c}}
\newcommand{\rsqarr}{\end{array} \right]}

\newcommand{\seabox}{{\framebox{$\searrow$}}}
\newcommand{\neabox}{{\framebox{$\nearrow$}}}
\newcommand{\swabox}{{\framebox{$\swarrow$}}}
\newcommand{\nwabox}{{\framebox{$\nwarrow$}}}

\newcommand{\inv}{\varprojlim}

\newtheorem{theorem}{Theorem}
\newtheorem{definition}{Definition}
\newtheorem{lemma}{Lemma}
\newtheorem{corollary}{Corollary}
\newtheorem{proposition}{Proposition}
\begin{document}
\title{Asymptotic structure in substitution tiling spaces}
\author{Marcy~Barge and Carl Olimb}


\begin{abstract} Every sufficiently regular space of tilings of $\R^d$ has at least one pair of distinct tilings that are asymptotic under translation in all the directions of some open $(d-1)$-dimensional hemisphere.
If the tiling space comes from a substitution, there is a way of defining a location on such tilings at which
asymptoticity `starts'. This leads to the definition of the {\em branch locus} of the tiling space: this is a subspace of the tiling space, of dimension at most $d-1$, that summarizes the `asymptotic in at least a half-space' behavior in the tiling space. We prove that if a $d$-dimensional self-similar substitution tiling space has a pair of distinct tilings that are asymptotic in a set of directions that contains a closed $(d-1)$-hemisphere in its interior, then the branch locus is a topological invariant of the tiling space. If the tiling space is a 2-dimensional self-similar Pisot substitution tiling space, the branch locus has a description as an inverse limit of an expanding Markov map on a 1-dimensional simplicial complex.
\end{abstract}

\maketitle

\section{Introduction}\label{sec:intro}

Tilings of $\R^d$, and the tiling spaces associated with them, play a fundamental role in recent investigations in number theory, physics, logic, computer science, and dynamical systems. For a given tiling $T$ of $\R^d$, one would like to understand the recurrence properties of the various patterns made by the tiles of $T$. These properties are encoded in the topology of the tiling space $\Omega$
associated with $T$ and in the dynamics of the action of $\R^d$ by translation on $\Omega$. Under standard assumptions on the nature of $T$, the space $\Omega$ is locally homeomorphic with a product of a $d$-dimensional disk and a Cantor set, the arc-components of $\Omega$ coincide with the orbits of the $\R^d$-action, and recurrences of patterns in a tiling correspond to returns under the $\R^d$-action to local neighborhoods.

We are interested here in the {\em asymptotics} of the translation action on tiling spaces. We will put sufficient conditions on the tilings we consider so that the associated tiling spaces are compact metric spaces: under the metric, $d$, two tilings are close if a small translate of one agrees with the other in a large neighborhood of the origin. Tilings $T,T'\in\Omega$ are then asymptotic in direction $v\in \mathbb{S}^{d-1}$ if $d(T-tv,T'-tv)\to 0$ as $t\to\infty$. We prove under rather general hypotheses that there are
$T\ne T'\in\Omega$ that are asymptotic in an entire open hemisphere's worth of directions. To say more, we must restrict a bit.

There are three common procedures for constructing tilings of $\R^d$: cut-and-project; matching rules; and substitutions. With the cut-and-project method, the asymptotic tilings and directions can be immediately deduced from the geometry of the boundary of the defining window (see, for example, \cite{FHK}). Global properties of matching rule tilings are, on the other hand, very difficult to ascertain - unless the rules enforce the sort of hierarchical structure enjoyed by substitution tilings (see \cite{M}, \cite{G-S}). We are thus led to consider substitutions $\Phi$ and their associated tiling spaces $\OP$.

A substitution $\Phi$ induces a self-map, also denoted by $\Phi$, of $\OP$. This will, among other things, provide a way of pinning down a specific location for the start of asymptoticity. For example, in the
one-dimensional ($d=1$) case, there are a finite and non-zero number of asymptotic $\R$-orbits
(at most $n^2$ if there are $n$ distinct tiles - see \cite{BDH}). Each of these orbits contains a unique tiling that is periodic under $\Phi$ and we may reasonably designate these $\Phi$-periodic, $\R$-asymptotic, orbits as the exact locations on their arc-components where asymptoticity starts. Let us say this more precisely. If $T\ne T'\in\Omega$ are such that $d(T-t,T'-t)\to0$ as $t\to\infty$
(or as $t\to-\infty$), then there is a unique $t_0\in\R$ so that, if $T_0:=T-t_0$ and $T'_0:=T'-t_0$, then
$T_0$ and $T'_0$ are $\Phi$-periodic. Moreover, $T_0\notin W^s(T'_0)$ and $T_0-t\in W^s(T'_0-t)$ for all $t>0$ (resp., $t<0$), where $W^s$ denotes the $\Phi$-stable `manifold'. This choice for initial point of asymptoticity is made less arbitrary by the rigidity result of \cite{BS}: if $\OP$ and $\Omega_{\Psi}$ are homeomorphic one-dimensional substitution tiling spaces, then there is a homeomorphism of
$\OP$ with $\Omega_{\Psi}$ that not only conjugates the $\R$-actions (perhaps with rescaling), but also
conjugates some power of $\Phi$ with a power of $\Psi$ and hence takes $\Phi$-periodic orbits to $\Psi$-periodic orbits. Thus, the finite collection of {\em asymptotic pairs} $\{(T_0,T_0')\}$, as above, is a topological invariant for $\OP$. (A complete topological invariant for one-dimensional substitution tiling spaces is derived from asymptotic pairs in \cite{BD1} and their connection with the Matsumoto $K_0$-group is explored in \cite{BDS}.)

In this article we extend the one-dimensional results to higher dimensional substitution tiling spaces.
We will call a pair of tilings $(T,T')\in\OP\times\OP$ a {\em branch pair} if $T\notin W^s(T')$ and there is an open hemisphere $H\subset \mathbb{S}^{d-1}$ so that $T-tv\in W^s(T'-tv)$ for all $v\in H$ and all $t>0$. We will see that such a pair is asymptotic in the directions $H$. (The terminology `branch pair' is inspired by the fact that, if the tiles are polyhedral, then the space $\OP$ is obtained as an inverse limit of a map on a branched manifold. The asymptotic-in-a-half-space pairs are then a ghostly remnant of the branching in the approximating branched manifolds.) The general situation is quite complicated when $d>1$. Members of branch pairs are not necessarily periodic under the substitution and there are typically infinitely many branch pairs. However, there are at most finitely many branch pairs that are asymptotic in a set of directions that contains a given closed hemisphere in its interior, and they are all substitution-periodic (Proposition \ref{periodic branch pairs exist}).
To capture the asymptotic structure of a substitution tiling space, we define the {\em branch locus} (essentially, this consists of all the members of branch pairs, together with all the tilings obtained by translating these members in directions in the boundary of their maximal asymptotic sectors, and all tilings accumulated on by such). We prove (Theorem \ref{topological invariance 1}) that if a self-similar substitution tiling space has a branch pair whose set of asymptotic directions contains a closed hemisphere in its interior, then the branch locus is a topological invariant.

If we restrict further to $2$-dimensional self-similar substitution tiling spaces, the structure of the branch locus becomes manageable: there are only finitely many lines along which branching occurs
(Theorem \ref{finitely many sectors}) and if the substitution is also Pisot, then the branch locus is the inverse limit of an expanding Markov map on a compact $1$-dimensional simplicial complex. In the latter case, it follows from the rigidity result of Kwapisz (\cite{Kwapisz}) that the branch locus is also a topological invariant (Theorem \ref{topological invariance 2}).

It frequently occurs that asymptotic structure in a tiling space can be held responsible for the appearance of certain subgroups in the cohomology of the space. For example, in the $1$-dimensional period doubling
tiling space ($a\mapsto ab, b\mapsto aa$), there is a branch pair $(T,T')$ with $d(T-t,T'-t)\to0$ as $t\to\pm\infty$: this pair contributes a $\Z$ in $H^1$. Similarly, the branch locus of the $2$-dimensional half-hex tiling space (see Example 1) consists of three
tilings, each asymptotic to the other two in all directions: this contributes $\Z^2$ to $H^2$.
For the chair tiling (Example 2), there are two ``tubes" made each of $\R^2$ orbits of four $2$-addic solenoids in the branch locus that contribute $\Z[1/2]^2$ to $H^2$ (the asymptotic structure also is responsible for a $\Z/3\Z$ that, while it doesn't show up in $H^2$ of the full tiling space, is, none-the-less, a topological invariant - see \cite{BDHS}). Calculations along these lines can be found in  the dissertation, \cite{O}, of  the second author, as can preliminary versions of many of the results contained in this article.

The existence of tilings that are asymptotic in a hemisphere's worth of directions is closely related to lack of expansiveness of translation subdynamics.
In \cite{BL}, Boyle and Lind consider subdynamics of expansive $\Z^d$ actions on compact metric spaces. In particular, they show that (for $d>1$) there is always a $d-1$-dimensional subspace of $\R^d$ so that the restriction of the $\Z^d$ action to directions within a bounded distance of the subspace is not expansive. The corresponding notion of expansiveness for the translation action of $\R^d$ on a $d$-dimensional tiling space $\Omega$ would be: the action is {\em transversely expansive} provided there is an $\alpha>0$ so that if $T\ne T'\in\Omega$ are any two tilings that share a tile, then there is a $v\in\R^d$ with $d(T-v,T'-v)\ge\alpha$. Under the assumption of finite local complexity (see Section \ref{sec:background}), the translation action is transversely expansive. Under additional (very mild) hypotheses, Proposition \ref{agree in half-space}
implies that there is always a $d-1$ dimensional subspace of $\R^d$ restricted to which the translation action on $\Omega$ is not transversely expansive. If $d=2$ and $\Omega$ is a self-similar substitution tiling space,
it is a consequence of Theorem \ref{two sectors} that there are at least two independent directions in which translation is not transversely expansive.

In Section \ref{sec:background} we review basic definitions and facts about tiling spaces. Section
\ref{sec:definitions} introduces the terminology related to asymptotic structures and establishes some general results. In Section \ref{2-d} we restrict to the case of $2$-dimensional self-similar substitution tiling spaces.

\section{Background}\label{sec:background}

By a {\em tile} $\tau$ in $\R^d$ we mean an ordered pair $\tau=(spt(\tau),m)$ where $spt(\tau)$, the {\em support} of $\tau$, is a compact subset of $\R^d$ and $m$ is a {\em mark} taken from some finite set of marks. A tile $\tau$ is {\em topologically regular} if $cl(int(spt(\tau)))=spt(\tau)$ and tiles $\tau=(spt(\tau),m)$ and $\sigma$ are {\em translationally equivalent} if there is a $v\in\R^d$ with $\tau+v:=((spt(\tau)+v,m)=\sigma$. In this article all tiles will be assumed to be topologically regular. By the {\em interior of a tile} we will mean the interior of its support: $\mathring{\tau}:= int(spt(\tau))$.

A {\em patch} is a collection of tiles with pairwise disjoint interiors, the {\em support} of a patch $P$, $spt(P)$, is the union of the supports of its constituent tiles, the {\em diameter} of $P$, $diam(P)$, is the diameter of its support, and a {\em tiling} of $\R^d$ is a patch with support $\R^d$. A collection $\Omega$ of tilings of $\R^d$
has {\em translationally finite local complexity} (FLC) if it is the case that for each $R$ there are
only finitely many translational equivalence classes of patches $P\subset T\in\Omega$ with $diam(P)\le R$. Given a tiling $T$, let $B_0[T]:=\{\tau\in T:0\in spt(\tau)\}$, and, for $R>0$, $B_R[T]:=\{\tau\in T:B_R(0)\cap spt(\tau)\ne\emptyset\}$. If $\Omega$ is a collection of tilings of $\R^d$ with FLC, there is a metric $d$ on $\Omega$ with the property: $d(T,T')<\epsilon$ if there are $v,v'\in\R^n$ with $|v|,|v'|<\epsilon/2$ so that $B_{1/\epsilon}[T'-v']=B_{1/\epsilon}[T-v]$. In other words, in this metric two tilings are close if a small translate of one agrees with the other in a large neighborhood of the origin. (See \cite{AP} for details.) We will call a collection $\Omega$ of tilings of $\R^d$ a {\em d-dimensional tiling space} if $\Omega$ has FLC, is closed under translation ($T\in\Omega$ and $v\in\R^d \Rightarrow T-v\in\Omega$), and is compact in the metric $d$. (All tiling spaces in this article are assumed to have FLC, but we will occasionally include the FLC hypothesis for emphasis.) For example, if $T$ is an FLC tiling of $\R^d$, then $\Omega=\{T':T'$ is a tiling of $\R^d$ and every patch of $T'$ is a translate of a patch of $T\}$ is a $d$-dimensional tiling space, called the {\em hull} of $T$ (\cite{AP}).

Suppose that $\mathcal{A}=\{\rho_1,\ldots,\rho_k\}$ is a set of translationally inequivalent tiles (called {\em prototiles}) in $\R^d$ and $\Lambda$ is an expanding linear isomorphism of $\R^d$. A {\em substitution} on $\mathcal{A}$ with expansion $\Lambda$ is a function $\Phi:\mathcal{A}\to\{P:P$ is a patch in $\R^d\}$
with the properties that, for each $i\in\{1,\ldots,k\}$, every tile in $\Phi(\rho_i)$ is a translate of an element of $\mathcal{A}$, and $spt(\Phi(\rho_i))=\Lambda(spt(\rho_i))$. Such a substitution naturally extends to patches whose elements are translates of prototiles by $\Phi(\{\rho_{i(j)}+v_j:j\in J\}):=\cup_{j\in J}(\Phi(\rho_{i(j)})+\Lambda v_j)$. A patch $P$ is {\em allowed} for $\Phi$ if there is an $m\ge1$, an $i\in\{1,\ldots,k\}$, and a $v\in\R^d$, with $P\subset \Phi^m(\rho_i)-v$. The {\em substitution tiling space} associated with
$\Phi$ is the collection $\OP:=\{T:T$ is a tiling of $\R^d$ and every finite patch in $T$ is allowed for $\Phi\}$. If $\Lambda$ is multiplication by the scalar $\lambda>1$, $\OP$ is called {\em self-similar}.

A $d$-dimensional tiling space $\Omega$ is {\em repetitive} if for each patch $P$ with compact support that occurs in some tiling in $\Omega$ there is an $R$ so that for all $T\in\Omega$ and all $x\in\R^d$, there is a $v\in\R^d$ so that $P-v\subset B_R[T-x]$. It is clear that if $\Omega$ is repetitive, then the action of $\R^d$ on $\Omega$ by translation is minimal. The substitution $\Phi$ is {\em primitive} if for each pair $\{\rho_i,\rho_j\}$ of prototiles there is an $n\in\N$ so that a translate of $\rho_i$ occurs in $\Phi^n(\rho_j)$.  If $\Phi$ is primitive then $\OP$ is repetitive.

If the translation action on $\Omega$ is free (i.e., $T-v=T\Rightarrow v=0$), $\Omega$ is said to be {\em non-periodic}. If $\Phi$ is primitive and $\OP$ is FLC and non-periodic then $\OP$ is compact in the metric described above, $\Phi:\OP\to\OP$ is a homeomorphism, and the translation action on $\OP$ is minimal and uniquely ergodic ( \cite{AP}, \cite{Solo}, \cite{sol}).

A real number $\lambda$ is a {\em Pisot number} if it is an algebraic integer and all of its algebraic conjugates lie strictly inside the unit circle. That is, there is a monic integer polynomial $p$, the {\em minimal polynomial} of $\lambda$, that is irreducible over $\Q$, has $\lambda$ as a root, and all other roots of $p$ have absolute value less than 1. The {\em degree} of $\lambda$ is the degree of $p$. A self-similar substitution is a {\em Pisot substitution} if the expansion $\lambda$ for $\Phi$ is a Pisot number.

\section{Definitions and preliminary results}\label{sec:definitions}

Suppose that $\Omega$ is a $d$-dimensional tiling space and $v\in\mathbb{S}^{d-1}$. Tilings
$T,T'\in\Omega$ are {\em asymptotic in direction $v$} provided $d(T-tv,T'-tv)\to 0$ as $t\to\infty$.
Given $S\subset\mathbb{S}^{d-1}$, $T$ and $T'$ are {\em uniformly asymptotic in
directions $S$} provided for each $\epsilon>0$ there is an $R$ so that $d(T-tv,T'-tv)<\epsilon$ for all
$v\in S$ and all $t\ge R$. We'll say that $T$ and $T'$ {\em agree on $X\subset\mathbb{R}^d$} provided
$B_0[T-v]=B_0[T'-v]$ for all $v\in X$.

\begin{proposition}\label{agree in half-space} If $\Omega$ is a non-periodic, repetitive, FLC tiling space
then there are tilings $\bar{T}\ne \bar{T}'$ in $\Omega$ that agree on an open half-space.
\end{proposition}
\begin{proof} Let $T$ be any tiling in $\Omega$. By repetitivity, there are $x_n\ne 0$ with $|x_n|\to\infty$
so that $B_n[T-x_n]=B_n[T]$. By non-periodicity, $T-x_n\ne T$ so $r_n:=sup\{r:B_r[T-x_n]=B_r[T]\}<\infty$. Pick $y_n\in B_{r_n}(0)$ with $B_1[T-x_n-y_n]\ne B_1[T-y_n]$. Since $B_1[T-x_n-y_n]\cap B_1[T-y_n]\ne\emptyset$, finite local complexity insures that, up to translation, there are only finitely many pairs
$(B_1[T-x_n-y_n], B_1[T-y_n])$. Passing to a subsequence we may assume that $(B_1[T-x_n-y_n], B_1[T-y_n])=(P'-z_n,P-z_n)$, with $z_n\to 0$, for some pair of distinct patches $(P',P)$. By compactness, there is a subsequence $n_i$,  tilings $\bar{T}',\bar{T}\in\Omega$, and $u$, so that $T-x_{n_i}-y_{n_i}\to\bar{T}'$, $T-y_{n_i}\to\bar{T}$, and $y_{n_i}/|y_{n_i}|\to u$. Then $\bar{T}\ne\bar{T}'$ (since $\bar{T}'\supset P',\bar{T}\supset P$) and
$\bar{T},\bar{T}'$ agree on the half-space $\{v:\langle v,u\rangle<0\}$.
\end{proof}

The translation action on a tiling space is not expansive: if $|w|$ is small, then $d(T-v,T-w-v)$ remains small for all $v$. But, assuming FLC, the $\R^d$-action on a $d$-dimensional tiling space is {\em transversely expansive} in the following sense:
there is an $\alpha>0$ so that if $T\ne T'\in\Omega$ and $T\cap T'\ne\emptyset$ then there is a $v\in\R^d$ so that $d(T-v,T'-v)\ge\alpha$. This notion of transversely expansive is the natural way of transferring expansivity of a $\Z^d$-action to ``expansivity" of the $\R^d$-action obtained by suspension. More generally, if $V$ is a subspace of $\R^d$, then the translation action on $\Omega$ is {\em transversely expansive in the direction of V} if, in the definition of transversely expansive, the $v$ can be taken from $V$. The following is analogous to a result of Boyle and Lind (\cite{BL}) for expansive $\Z^d$-actions.

\begin{corollary}\label{transversely expansive} If $\Omega$ is a $d$-dimensional, non-periodic, repetitive, FLC tiling space then there is a $d-1$-dimensional subspace $V$ of $\R^d$ so that the translation action is not transversely expansive in the direction of $V$. Moreover, if $V$ is such a subspace, then there are tilings $\bar{T}\ne\bar{T}'\in\Omega$ that agree on one of the components of
$\R^d\setminus V$.
\end{corollary}
\begin{proof} By Proposition \ref{agree in half-space} there is a $w\ne 0$ and tilings $\bar{T}\ne\bar{T}'\in\Omega$ so that $\bar{T}$ and $\bar{T}'$ agree on $\{x:\langle x,w\rangle>0\}$. Let $V:=\{w\}^{\perp}$.
Then, for any $\alpha>0$, there is a $t>0$ so that, if $T:=\bar{T}-tw$ and $T':=\bar{T}'-tw$, $sup\{d(T-v,T'-v):v\in V\}<\alpha$.

Suppose now that $V$ is a $d-1$-dimensional subspace of $\R^d$ so that the translation action on $\Omega$ is not transversely expansive in the direction of $V$. There are then $T_n\ne T_n'\in\Omega$
so that $T_n\cap T_n'\ne\emptyset$ and $d(T_n-v,T_n'-v)<1/n$ for each $n\in\N$ and all $v\in V$. There are then $x_n\to0$ so that $B_0[T_n-x_n]=B_0[T_n']$ for all $n\in\N$. It follows that there are $t_n\to\infty$ so that $T_n-x_n -v$ agrees with $T_n'-v$ on the strip $V+\{tw:|t|\le t_n\}$, where $w\in\mathbb{S}^{d-1}$ is perpendicular to $V$. Notice that since $T_n\cap T_n'\ne\emptyset$, $T_n-x_n\ne T_n'$ for large $n$. By appropriately translating $T_n-x_n$ and $T_n'$ we obtain tilings $S_n\ne S_n'\in\Omega$ with the properties:  $B_1[S_n]\ne B_1[S_n']$ and  $S_n$ and $S_n'$ agree on $V+\{tw:0\le t\le 2t_n\}$ or on
 $V-\{tw:0\le t\le 2t_n\}$. Let $n_i\to\infty$ be such that $S_{n_i}\to\bar{T}\in\Omega$ and $S_{n_i}'\to\bar{T}'\in\Omega$. Then $\bar{T}\ne\bar{T'}$ and $\bar{T}$ and $\bar{T}'$ agree on one of the components of $\R^d\setminus V$.
\end{proof}

It is clear that if tilings $T$ and $T'$ agree on the open half-space $\{v\in\mathbb{R}^d:\langle v,u\rangle\,>0\}$, some $u\in\mathbb{S}^{d-1}$, then $T$ and $T'$ are uniformly asymptotic in directions $\{v\in\mathbb{S}^{d-1}:\langle v,u\rangle\,>\epsilon\}$ for each $\epsilon>0$.

\begin{lemma}\label{asymptotic direction} If $\Omega$ is a $d$-dimensional FLC tiling space and $T,T'\in\Omega$ are asymptotic in direction $v\in\mathbb{S}^{d-1}$ then there is $R$ so that $T$ and $T'$ agree on $\{tv:t\ge R\}$.
\end{lemma}
\begin{proof} As a consequence of finite local complexity, there is $\epsilon>0$ so that: (1) if $T_1,T_2\in\Omega$ satisfy $B_1[T_1]=B_1[T_2]$ and $|x|<\epsilon$, then $B_0[T_1-x]=B_0[T_2-x]$; and (2) if $B_0[T_1]=B_0[T_1-x]$ for some $T_1\in\Omega$ and $x$ with $|x|<\epsilon$, then $x=0$. Now if
$T$ and $T'$ are asymptotic in direction $v$, there is $x(t)$ so that $B_1[T-tv]=B_1[T'-tv-x(t)]$ and $x(t)\to0$ as $t\to\infty$. Let $R$ be large enough so that $|x(t)|<\epsilon/2$ for $t\ge R$. For $t_0>R$ we have
$B_1[T-t_0v]=B_1[T'-t_0v-x(t_0)]$ so $B_0[T-tv]=B_0[T'-tv-x(t_0)]$ for all $t$ with $|t-t_0|<\epsilon$, by (1). Also, $B_0[T-tv]=B_0[T'-tv-x(t)]$. Thus, if $|t-t_0|<\epsilon$ and $t,t_0\ge R$, we have $B_0[T'-tv-x(t_0)]=B_0[T'-tv-x(t)]$ and $|x(t)-x(t_0)|<\epsilon$, so that $x(t)=x(t_0)$, by (2). That is, $x(t)$ is constant
for $t\ge R$, and hence $x(t)=0$ for $t\ge R$.
\end{proof}

\begin{lemma}\label{asymptotic directions}  If $\Omega$ is a $d$-dimensional FLC tiling space and $T,T'\in\Omega$ are uniformly asymptotic in directions $S\subset\mathbb{S}^{d-1}$ then there is $R$ so that $T$ and $T'$ agree on $\{tv:v\in S,t\ge R\}$.
\end{lemma}
\begin{proof} Let $\epsilon$ be as in the proof of Lemma \ref{asymptotic direction}. Then, for each
$v\in S$ there is an $R$ and $x(t,v)$ with $|x(t,v)|<\epsilon/2$ for $t\ge R$ so that $B_1[T-tv]=
B_1[T'-tv-x(t,v)]$. As in the proof of Lemma \ref{asymptotic direction}, this implies that $x(t,v)=0$
for all $v\in S$ and $t\ge R$. Thus $T$ and $T'$ agree on $\{tv:v\in S,t\ge R\}$.
\end{proof}

If $\Phi$ is the induced homeomorphism on a substitution tiling space $\OP$ and $T\in\OP$,
the {\em stable manifold} of $T$ is the set $W^s(T):=\{T'\in\OP:d(\Phi^n(T'),\Phi^n(T))\to0$ as $n\to\infty\}$.

\begin{lemma}\label{stable mfd} Suppose that $\Omega=\Omega_{\Phi}$ is an FLC substitution tiling space and $T\in\Omega$. Then $W^s(T)=\{T'\in\Omega:\forall r,\exists n$ such that $B_r[\Phi^n(T')]=B_r[\Phi^n(T)]\}=\{T'\in\Omega:B_0[\Phi^n(T')]=B_0[\Phi^n(T)]$ for some $n\}$.
\end{lemma}
\begin{proof} The equality of the last two sets, and their containment in $W^s(T)$, is clear.
For the opposite containment, note that, by finite local complexity, there is $\epsilon_1>0$ so that, for all
$\epsilon_2$, with $0<\epsilon_2\le\epsilon_1$, $inf\{d(T_1,T_2-x):T_1,T_2\in\Omega, B_1[T_1]=B_1[T_2],\epsilon_2\le|x|\le\epsilon_1\}=:\delta(\epsilon_2)>0$. Choose $\epsilon_1>0$ also small enough so that if $T_1$ and $T_2$ are any two tilings in $\Omega$ with $B_1[T_1]=B_1[T_2]$, then $B_1[T_1-x]=B_1[T_2-y]$ and $|x|,|y|\le\epsilon_1$ implies that $x=y$. Set $\epsilon_2=\epsilon_1/|\Lambda|_{max}$, where $|\Lambda|_{max}:=max\{|\Lambda v|:|v|=1\}$ and $\Lambda$ is the inflation matrix for $\Phi$. Now if  $T'\in W^s(T)$, there are $N$  and $x_n$ with $|x_n|<\epsilon_2$, so that $B_1[\Phi^n(T')]=B_1[\Phi^n(T)-x_n]$ and $d(\Phi^n(T'),\Phi^n(T))<\delta(\epsilon_2)$ for all $n\ge N$. If $x_N\ne0$, there is $k>0$ so that $\epsilon_2\le |\Lambda^k x_N|\le\epsilon_1$. We would have $B_1[\Phi^{N+k}(T)-x_{N+k}]=B_1[\Phi^{N+k}(T')]=B_1[\Phi^k(\Phi^N(T)-x_N)]=B_1[\Phi^{N+k}(T)-\Lambda^kx_N]$ so that $x_{N+k}=\Lambda^kx_N$, from which it would follow that $d(\Phi^{N+k}(T'),\Phi^{N+k}(T))\ge\delta(\epsilon_2)$, in contradiction to the choice of $N$. Thus $B_1[\Phi^N(T')]=B_1[\Phi^N(T)]$.
\end{proof}

If $\Omega=\OP$ is a $d$-dimensional substitution tiling space with $T,T'\in\Omega$ and $T\in W^s(T')$ we will say that
$T$ and $T'$ are {\em stably equivalent} and write $T\sim_s T'$. Note that as a consequence of Lemma \ref{stable mfd} the relation $\nsim_s$, the negation of $\sim_s$, has the property: $T-x_n\nsim_s T'-x_n$ and $x_n\to x$ $\implies$ $T-x\nsim_s T'-x$ . The tiling $T\in\Omega$
will be called a {\em branch point} if there is a tiling $T'\in\Omega$ and an open half space $H$ so that
$T\nsim_s T'$ and $T-x\sim_s T'-x$ for all $x\in H$. In this case, $(T,T')$ will be called a {\em branch pair};
the connected component of $\{v\in\mathbb{S}^{d-1}:T-tv\sim_sT'-tv,\,\forall t>0\}$ containing $H\cap\mathbb{S}^{d-1}$ will
be called the {\em asymptotic sector} of $(T,T')$ and denoted by $S(T,T')$. We will denote the collection of all branch pairs of $\OP$ by $\mathcal{BP}(\Phi)$. (It may occur that a branch pair has two different asymptotic sectors, $H\cap\mathbb{S}^{d-1}$ and $-H\cap\mathbb{S}^{d-1}$ for some open half-space $H$, but this ambiguity will be harmless.)

For a given substitution tiling space $\Omega$, patches $P,P'$, and $X\subset int(spt(P)\cap spt(P'))$, we'll say that $P$ and $P'$ are {\em stably related on $X$} provided, whenever $T,T'\in\Omega$ are such that $P\subset T,P'\subset T'$, it is the case that $T-x\sim_s T'-x$ for all $x\in X$. If $P$ and $P'$ are stably related on $X=int(spt(P)\cap spt(P'))\ne \emptyset$, we'll say that $P$ and $P'$ are {\em stably related on overlap}. The following is an adaptation of a lemma in \cite{sol}.

\begin{lemma}\label{no slip} If $P$ and $P'$ are patches for a non-periodic, FLC, self-similar substitution tiling space and $P$ and $P'$ are stably related on overlap, there is an $\epsilon>0$ so that if $0<|x|<\epsilon$ then $P-x$ and $P'$ are not stably related on overlap.
\end{lemma}
\begin{proof} We prove the result in case $P=P'=\{\tau\}$ for some tile $\tau$, from which the general result follows easily. Let $\lambda$ be the expansion factor for $\Phi$. Pick $y\in int(\tau)$ and let $r>0$ be small enough so that $cl(B_{2r}(y))\subset int(\tau)$. By Lemma 2.4 of \cite{sol}, there is an $N>1$ so that if $Q$ and $Q-x_1$ are both patches in some tiling $T\in\Omega$, $B_{r_1}(y)\subset spt(Q)$, and $|x_1|<r_1/N$, then $x_1=0$.
Let $\epsilon:=r/N$, let $|x|<\epsilon$, and suppose that $\{\tau\}$ and $\{\tau\}-x$ are stably related on overlap. Since the compact ball $cl(B_r(y))$ is contained in $int(\tau\cap( \tau-x))$, there is an $M$ so that $B_0[\Phi^M(\{\tau\}-z)]=B_0[\Phi^M(\{\tau\}-x-z)]$ for all $z\in cl(B_r(y))$ (Lemma \ref{stable mfd}). That is, both the patch $Q:=B_{\lambda^Mr}[\Phi^M(\{\tau\})]$ and the patch $Q-\lambda^Mx=
B_{\lambda^Mr}[\Phi^M(\{\tau\}-x)]$ occur in the patch $\Phi^M(\{\tau\})$. Moreover, letting $x_1=\lambda^Mx$ and $r_1=\lambda^Mr$, we have $B_{r_1}(y)\subset spt(Q)$ and $|x_1|<r_1/N$. Thus $x_1=0$ and hence $x=0$.
\end{proof}

\begin{lemma}\label{T=T'} If $\Omega=\Omega_{\Phi}$ is a $d$-dimensional non-periodic, FLC,
self-similar substitution tiling space and $T,T'\in\Omega$ are such that $T-x\sim_sT'-x$ for all $x\in\mathbb{R}^d$, then $T=T'$.
\end{lemma}
\begin{proof} Let $\Delta:=max\{diam(\rho):\rho$ is a prototile for $\Omega\}$. Then for all $\bar{T}\in\Omega$, $spt(B_0[\bar{T}])\subset int(spt(B_{\Delta+1}[\bar{T}]))$, and there are, up to translation, just finitely
many pairs of the form $(B_0[\bar{T}],B_{\Delta+1}[\bar{T}])$. It follows from Lemma \ref{no slip} that, up to translation, there are only finitely many pairs of pairs of the form $((B_0[T-x],B_{\Delta+1}[T-x]),(B_0[T'-x],B_{\Delta+1}[T'-x]))$, say for each $x$ there is an $i\in\{1,\ldots,n\}$ and a $y=y(x)$ with $((B_0[T-x],B_{\Delta+1}[T-x]),(B_0[T'-x],B_{\Delta+1}[T'-x]))=((P_i-y,Q_i-y),(P'_i-y,Q'_i-y))$. Let $x,y,i$ be as in the last sentence. For each $z\in spt(P_i)\cap
spt(P'_i)$, there is an $m=m(x,y,z)$ so that $B_1[\Phi^m(T-x+y-z)]=B_1[\Phi^m(T'-x+y-z)]$, by Lemma \ref{stable mfd}, and in fact we can choose this $m$ to depend only on $z$ (and $i$), since if $x',y',i$ are also such that  $((B_0[T-x'],B_{\Delta+1}[T-x']),(B_0[T'-x'],B_{\Delta+1}[T'-x']))=((P_i-y',Q_i-y'),(P'_i-y',Q'_i-y'))$, then
$(B_0[T-x'+y'-z],B_0[T'-x'+y'-z])=(B_0[T-x+y-z],B_0[T'-x+y-z]))$ (this is because $spt(P_i)\cap spt(P'_i)\subset int(spt(Q_i)\cap spt(Q'_i))$). Now if $z'$ is sufficiently close to $z$, then $B_0[\Phi^m(T-x+y-z')]=B_0[\Phi^m(T'-x+y-z')]$. Thus, by compactness of $spt(P_i)\cap spt(P'_i)$, there is an $M=M(i)$ so that
 $B_0[\Phi^M(T-x+y-z)]=B_0[\Phi^M(T'-x+y-z)]$ for all $z\in spt(P_i)\cap spt(P'_i)$ and $x,y$ so that
 $((B_0[T-x],B_{\Delta+1}[T-x]),(B_0[T'-x],B_{\Delta+1}[T'-x]))=((P_i-y,Q_i-y),(P'_i-y,Q'_i-y))$ Let $K=max\{M(1),\ldots,M(n)\}$. Then $B_0[\Phi^K(T-x+y-z)]=B_0[\Phi^K(T'-x+y-z)]$ for all appropriate $x,y,z$: that is, $\Phi^K(T)=\Phi^K(T')$, and since $\Phi$ is a homeomorphism, $T=T'$.
\end{proof}

\begin{lemma}\label{existence of R} If $\Omega=\Omega_{\Phi}$ is a $d$-dimensional non-periodic, FLC, self-similar substitution tiling space, there is an $R$ so that if $T,T'\in\Omega$ are such that $T-x\sim_s T'-x$ for all $x\in\mathbb{R}^d$ with $|x|\le R$, then $B_0[T]=B_0[T']$.
\end{lemma}
\begin{proof} Otherwise, for each $n\in\N$, there are $T_n,T'_n\in\Omega$ with $T_n-x\sim_s T'_n-x$ for all $x$ with $|x|\le n$ and $B_0[T_n]\ne B_0[T'_n]$. There is a smallest $m_n$ so that $B_0[\Phi^{m_n}(T_n)]=B_0[\Phi^{m_n}(T'_n)]$: replace $T_n$ by $\Phi^{m_n-1}(T_n)$ and $T'_n$ by $\Phi^{m_n-1}(T'_n)$. By finite local complexity, there are only finitely many pairs $(B_0[T_n],B_0[T'_n])$, up to translation: say $(B_0[T_{n_i}],B_0[T'_{n_i}])=(P-x_i,P'-x_i)$ with $x_i\to 0$, $T_{n_i}\to T\in\Omega$,
and $T'_{n_i}\to T'\in\Omega$. Since $P\subset T$ and $P'\subset T'$, $T\ne T'$. Fix $r>0$. Since the
tilings $T_{n_i}+x_i$ all share the patch $P$ and the tilings $T'_{n_i}+x_i$ all share the patch $P'$, finite local complexity guarantees that there are just finitely many pairs $(B_r[T_{n_i}+x_i],B_r[T'_{n_i}+x_i])$. It follows
that $B_r[T_{n_i}+x_i]\subset T$ and $B_r[T'_{n_i}+x_i]\subset T'$ for large $i$. If $i$ is also large enough so that $n_i+|x_i|>r$, then $T_{n_i}+x_i-y\sim_sT'_{n_i}+x_i-y$ for all $y\in B_r(0)$ and hence
$T-y\sim_s T'-y$ for all $y\in B_r(0)$. Thus $T-y\sim_s T'-y$ for all $y$ and by Lemma \ref{T=T'} we arrive at the contradiction $T=T'$.
\end{proof}

\begin{lemma}\label{finitely many patch pairs} If $\Omega=\Omega_{\Phi}$ is a $d$-dimensional, non-periodic, FLC, self-similar substitution tiling space, then given $r$ and $\epsilon>0$, there are,
up to translation, only finitely many pairs of (allowed) patches $(P,P')$ with the properties: $diam(P),diam(P')\le r$; $P$ and $P'$ are stably related on $B_\epsilon(0)$.
\end{lemma}
\begin{proof} First, there is a $k=k(\epsilon)$ so that for any such $(P,P')$, $\Phi^k(P)\cap \Phi^k(P')\ne\emptyset$. Otherwise, there are $(P_k,P'_k)$ and tilings $T_k\supset \Phi^k(P_k),T'_k\supset \Phi^k(P'_k)$ with $T_k-x\sim_sT'_k-x$ for all $x\in B_{\lambda^k\epsilon}(0)$ and with $B_0[T_k]\ne B_0[T'_k]$. This violates Lemma \ref{existence of R}. Now, for this $k$ and for all such pairs $(P,P')$, the elements of the pairs $(\Phi^k(P),\Phi^k(P'))$ share a tile and have diameters bounded by $\lambda^kr$. By finite local complexity, there are, up to translation, only finitely many such $(\Phi^k(P),\Phi^k(P'))$ and hence, only finitely many $(P,P')$.
\end{proof}

\begin{proposition}\label{finitely many asymptotic partners} If $\Omega$ is a non-periodic, FLC, self-similar substitution tiling space and $T\in\Omega$ is a branch point then there are finitely many $T'\in\Omega$ such that $(T,T')$ is a branch pair.
\end{proposition}
\begin{proof} Let $\Phi$ be the substitution with associated inflation $\lambda$. Suppose $(T,T_k')$, $k\in\mathbb{N}$ are branch pairs. It follows from Lemma \ref{existence of R} and finite local complexity that there is a subsequence $\{T'_{k_i}\}$ and an $R$ so that, for each $m\in\mathbb{N}$, the $B_R[\Phi^{-m}(T'_{k_i})]$
all share a tile. There is $M$ so that if $\{T_1,\ldots,T_{M}\}$ are any $M$ tilings that satisfy $\cap_{i+1}^{M}B_R[T_i]\ne\emptyset$, then $B_1[T_i]=B_1[T_j]$ for some $i\ne j\in\{1,\ldots,M\}$. If the $T'_{k_i}$, $i=1,\ldots,M$, are all distinct, then there is $r$ so that the patches $B_r[T'_{k_i}]$, $i=1,\ldots,M$ are all distinct. Let $m\in\mathbb{N}$ be large enough so that $\lambda^m>r$. There are then $i\ne j\in\{1,\ldots,M\}$ so that $B_1[\Phi^{-m}(T'_{k_i})]=B_1[\Phi^{-m}(T'_{k_j})]$. But then $B_r[T'_{k_i}]=B_r[T'_{k_j}]$, contradicting the choice of $r$.
\end{proof}

\begin{corollary}\label{periodic branch points} If $(T,T')$ is a branch pair and $T$ is $\Phi$-periodic, then so is $T'$.
\end{corollary}

In case $T$ is a branch point and $T$ is $\Phi$-periodic, $T$ will be called a {\em periodic branch point}, a branch pair $(T,T')$ will be called a {\em periodic branch pair} and we'll denote the collection of all periodic branch pairs by $\mathcal{PBP}(\Phi)$.

\begin{lemma}\label{sector open} Suppose that $\Omega=\Omega_{\Phi}$ is a $d$-dimensional, non-periodic, FLC, self-similar substitution tiling space. If $(T,T')\in\mathcal{PBP}(\Phi)$ is a periodic branch pair then $S(T,T')$ is open in $\mathbb{S}^{d-1}$.
\end{lemma}
\begin{proof}  With no loss of generality, we may assume that $T$ and $T'$ are fixed by $\Phi$. Let $\lambda$ be the expansion factor for $\Phi$. Given $v\in S(T,T')$ and $t\in [1,\lambda]$
there is $\epsilon(t)>0$ so that $T-tv-x\sim_s T'-tv-x$ for all $x\in B_{\epsilon(t)}(0)$ (Lemma \ref{stable mfd}). From compactness of $[1,\lambda]$, there is $\epsilon>0$ so that if $w\in \mathbb{S}^{d-1}\cap B_{\epsilon}(v)$ then $T-tw\sim_s T'-tw$ for all $t\in [1,\lambda]$. For such $w$, any $n\in\Z$, and $t\in [1,\lambda]$, $T-\lambda^ntw=\Phi^n(T-tw)\sim_s \Phi^n(T'-tw)=T'-\lambda^ntw$. Thus $T-tw\sim_s T'-tw$
for all $t\in \cup_{n\in\Z}\lambda^n[1,\lambda]=\R^+$. That is,  $\mathbb{S}^{d-1}\cap B_{\epsilon}(v)\subset S(T,T')$.
\end{proof}

\begin{lemma}\label{finitely many patches} Suppose that $\Omega=\Omega_{\Phi}$ is a non-periodic, FLC,  self-similar substitution tiling space. Then for each $r$ there are, up to translation, only finitely many pairs of patches
of the form $(B_r[T],B_r[T'])$ with $(T,T')\in\mathcal{BP}(\Phi)$. Furthermore, if $(T_n,T'_n)\in\mathcal{BP}(\Phi)$ and $(T_n,T'_n)\to (T,T')$ in $\Omega\times\Omega$, then, for sufficiently large $n$, there
are $x_n\to 0$ so that $(B_r[T_n],B_r[T'_n])=(B_r[T-x_n],B_r[T'-x_n])$.
\end{lemma}
\begin{proof}
The first statement follows immediately from Lemma \ref{finitely many patch pairs}.

Suppose that $(T_n,T'_n)\in\mathcal{BP}(\Phi)$ and $(T_n,T'_n)\to (T,T')$ in $\Omega\times\Omega$ and let $r$ be given.  There are then $x_n$ and $y_n$ so that $B_r[T_n]=B_r[T-x_n]$ and $B_r[T'_n]=B_r[T'-y_n]$ for all sufficiently large $n$, with $x_n,y_n\to 0$. If $x_n-y_n\to 0$ takes on infinitely many different values, we would have infinitely many translationally independent pairs $(B_r[T_n],B_r[T'_n])$.
Thus $x_n=y_n$ for sufficiently large $n$.
\end{proof}

\begin{lemma}\label{periodic tilings} Suppose that $\Omega=\Omega_{\Phi}$ is a non-periodic, FLC,  self-similar substitution tiling space and $T\in\Omega$ has the property that $B_r[\Phi^{-n}(T)]=B_r[\Phi^{-n_1}(T)]$ for some $r\ge0$,
$n_1\in\N$, and for infinitely many $n\in\N$. Then $T$ is $\Phi$-periodic.
\end{lemma}
\begin{proof} Say $B_r[\Phi^{-n_i}(T)]=B_r[\Phi^{-n_1}(T)]=:P$ for $i\in\N$. Let $m=min\{n\in\N:P\subset\Phi(P)\}$. Since $\Phi(B_r[T'])\supset B_{\lambda r'}[\Phi(T')]\supset B_r[\Phi(T')]$ for any $T'\in\Omega$, where $\lambda$ is the expansion factor of $\Phi$ and $r'=r$ if $r>0$ and, if $r=0$ then $r'>0$ is such that $B_{r'}(0)\subset spt(B_0[T'])$, we see that $m$ divides $n_i-n_1$ for all $i$. Say $n_i=n_1+mk_i$ and let $\bar{T}\in\Omega$ be defined by $\bar{T}:= \Phi^{n_1}(\cup_{k\in\N}\Phi^{mk}(B_r[\Phi^{-n_1}[T]))=\cup_{k\in\N}\Phi^{mk}(B_r[T])$. Then $\Phi^m(\bar{T})=\bar{T}$. Since the union in the definition of $\bar{T}$ is nested and $k_i\to\infty$, $\bar{T}= \Phi^{n_1}(\cup_{i\in\N}\Phi^{mk_i}(B_r[\Phi^{-n_1}(T)]))=
\Phi^{n_1}(\cup_{i\in\N}\Phi^{mk_i}(B_r[\Phi^{-n_1-mk_i}(T)]))\subset \Phi^{n_1}(\cup_{i\in\N}B_{\lambda^{mk_i}r'}[\Phi^{-n_1}(T)])\subset \Phi^{n_1}(\Phi^{-n_1}(T))=T$. Thus $T=\bar{T}$ is periodic.
\end{proof}

\begin{proposition}\label{periodic branch pairs exist} Suppose that $\Omega=\Omega(\Phi)$ is a $d$-dimensional, non-periodic, FLC,  self-similar substitution tiling space and $S\subset\mathbb{S}^{d-1}$ contains a
closed hemisphere in its interior, then any branch pair $(T,T')$ with
$S\subset S(T,T')$ is periodic and there are at most finitely many such.
\end{proposition}
\begin{proof}
Let $u\in\mathbb{S}^{d-1}$ be such that $\{v\in\mathbb{S}^{d-1}:\langle v,u\rangle\ge 0\}\subset S$, and let $\bar{H}=\{x\in\R^d:\langle x,u\rangle\ge 0\}$. Suppose that $(T_n,T'_n)$ are branch pairs with $S\subset S(T_n,T'_n)$
for $n\in \N$. By Lemma \ref{finitely many patches}, there are, up to translation, only finitely many pairs $(B_1[\Phi^k(T_n)],B_1[\Phi^k(T'_n)])$ for $n\in\N$ and $k\in\Z$. Suppose that $n_1,n_2,k_1,k_2$ and $x$ are such that $(B_1[\Phi^{k_1}(T_{n_1})],B_1[\Phi^{k_1}(T'_{n_1})])=(B_1[\Phi^{k_2}(T_{n_2})]-x,B_1[\Phi^{k_2}(T'_{n_2})]-x)$. The fact that $B_1[\Phi^{k_1}(T_{n_1})]$ is stably related to $B_1[\Phi^{k_1}(T'_{n_1})]$ on $int(spt(B_1[\Phi^{k_1}(T_{n_1})]))\cap int(spt(B_1[\Phi^{k_1}(T'_{n_1})]))\cap(\bar{H}\setminus \{0\})$ while $B_1[\Phi^{k_2}(T_{n_2})]-x$ is not stably related to $B_1[\Phi^{k_2}(T'_{n_2})]-x$ on $\{x\}$ means that
$x\notin \bar{H}\setminus\{0\}$. The preceding statement remains valid if $n_1$ and $n_2$ are switched, $k_1$ and $k_2$ are switched,
 and $x$ is replaced by $-x$. We have that $x\notin \bar{H}\setminus\{0\}$ and $-x\notin \bar{H}\setminus\{0\}$; hence $x=0$.

Now fix $n$. There must be $m>0$ and $l_i\to-\infty$ so that $(B_1[\Phi^{l_i}(T_n)], B_1[\Phi^{l_i}(T'_n)])$ and $(B_1[\Phi^{l_i+m}(T_n)], B_1[\Phi^{l_i+m}(T'_n)])$ are the same up to translation for each $i$. From the above,  $(B_1[\Phi^{l_i}(T_n)], B_1[\Phi^{l_i}(T'_n)])=(B_1[\Phi^{l_i+m}(T_n)], B_1[\Phi^{l_i+m}(T'_n)])$ for each $i$. By Lemma \ref{periodic tilings}, $T_n$ and $T'_n$ are periodic.

Since we have only finitely many distinct pairs $(B_1[T_n],B_1[T'_n])$ and since $B_1[T_n]=B_1[T_m]$ and $T_n,T_m$ periodic implies that $T_n=T_m$,
we have only finitely many branch pairs with asymptotic sector containing $S$, all of which are periodic.
\end{proof}

\begin{proposition}\label{branch points} If $\Omega=\Omega_{\Phi}$ is a $d$-dimensional, non-periodic, FLC, self-similar substitution tiling space and $T\ne T'\in\Omega$ agree on a set $X=\{v\in\mathbb{S}^{d-1}:\langle v,u\rangle\,>\alpha\}$ for some $u\in\mathbb{S}^{d-1}$ and $\alpha\le 0$, then there is a branch pair $(\bar{T},\bar{T}')\in\mathcal{BP}(\Phi)$
with $X\cap\mathbb{S}^{d-1}\subset S(\bar{T},\bar{T'})$. Moreover, if $\alpha<0$, we may take $\bar{T}=T-z$ and $\bar{T}'=T-z$ for some $z\in\R^d$.
\end{proposition}
\begin{proof}Let $\lambda$ be the expansion factor for $\Phi$. For each $n\in\mathbb{N}$, let $T_n:=\Phi^{-n}(T)$ and $T'_n:=\Phi^{-n}(T')$. Then $T_n-x\sim_s T'_n-x$ for all $x\in X$ and all $n\in\mathbb{N}$. It follows from Lemma \ref{existence of R}
that there is an $R>0$ so that $T_n$ and $T'_n$ agree on $\{Ru\}$ for all n. Just as in the proof of Lemma \ref{finitely many patches} there are, up to translation, only finitely many pairs of patches  $(B_r[T_n],B_r[T'_n])$ for any particular $r$.
Choose a subsequence $\{n_i\}$ so that $T_{n_i}\to \bar{T}\in\Omega$ and $T'_{n_i}\to \bar{T}'\in\Omega$ as $i\to\infty$. By Lemma \ref{T=T'}, there is a $y$ so that $T-y\nsim_s T'-y$. Then $T_{n_i}-y_{n_i}\nsim_s T'_{n_i}-y_{n_i}$ for all $i$ where $y_k:=\lambda^{-k}y$, since $T_{n_i}-y_{n_i}\to \bar{T}$ and $T'_{n_i}-y_{n_i}\to \bar{T}'$, $\bar{T}\nsim_s\bar{T}'$.
Suppose that $v$ is such that $\bar{T}-v\nsim_s\bar{T}'-v$. Let $r=|v|+1$. For large enough $i$ there are
$x_{n_i}$ and $x'_{n_i}$, with $|x_{n_i}|<1$, $|x'_{n_i}|<1$, $x_{n_i}\to 0$ and $x'_{n_i}\to 0$, so that $B_r[T_{n_i}]\subset B_{r+1}[\bar{T}]-x_{n_i}$ and
$B_r[T'_{n_i}]\subset B_{r+1}[\bar{T}']-x'_{n_i}$. Since there are only finitely many pairs $(B_r[T_{n_i}],B_r[T'_{n_i}])$, up to translation, there are arbitrarily large $i$ for which $x_{n_i}=x'_{n_i}$. For such $i$, $v-x_{n_i}\in int(spt(B_r[T_{n_i}]))\cap int(spt(B_r[T'_{n_i}]))$ so that $T_{n_i}-(v-x_{n_i})\sim_s T'_{n_i}-(v-x_{n_i})$ if and only if $\bar{T}-v\sim_s\bar{T}'-v$. Since $\bar{T}-v\nsim_s\bar{T}'-v$, $T_{n_i}-(v-x_{n_i})\nsim_s T'_{n_i}-(v-x_{N_i})$, and
hence $\langle v-x_{n_i},u\rangle\le \alpha$. Since $x_{n_i}\to 0$, $\langle v,u\rangle\le \alpha$. In other words, $\bar{T}-x\sim_s\bar{T}'-x$ for all $x\in X$.

Suppose now that $\alpha<0$. If $j_i\to\infty$ is any sequence so that $T_{j_i}\to\bar{T}_j$ and $T_{j_i}\to\bar{T}'_j$, the forgoing argument shows that $(\bar{T}_j,\bar{T}'_j)$ is a branch pair with $X\cap\mathbb{S}^{d-1}\subset S(\bar{T}_j,\bar{T}'_j)$. By Proposition \ref{periodic branch pairs exist}, there are only finitely many such branch pairs and they are all periodic. It must be the case that all these branch pairs are on the same periodic orbit under $\Phi\times\Phi$ (in general, if an $\alpha$-limit set is finite it must be a single periodic orbit). Say this periodic orbit is $\{(\Phi^k(\bar{T}),\Phi^k(\bar{T}')):k=1,\ldots,m\}$. There are then an $l$ and $x_n\to0$ so that $(B_1[T_n],B_1[T'_n])=(B_1[\Phi^{-n+l}(\bar{T})+x_n],B_1[\Phi^{-n+l}(\bar{T}')+x_n])$ for large $n$. From this we see that $x_n=\lambda x_{n+1}$ for sufficiently large $n$, say, for $n\ge N$.
We have $T_N\supset \Phi^{k}(B_1[T_{N+k }])=\Phi^{k}(B_1[\Phi^{-N-k+l}(\bar{T})+x_{N+k}])\supset B_{\lambda^k}[\Phi^{-N+l}(\bar{T})+x_N]$ for all $k\in\N$. That is, $T_N=\Phi^{-N+l}(\bar{T})+x_N$. Similarly, $T'_N=\Phi^{-N+l}(\bar{T'})+x_n$. Thus $(T-z,T'-z)=(\Phi^l(\bar{T}),\Phi^l(\bar{T}'))$ for $z=\lambda^Nx_N$.
\end{proof}

 \begin{lemma}\label{branch points with direction} Suppose that $\Omega=\Omega_{\Phi}$ is a $d$-dimensional, non-periodic, FLC, self-similar substitution tiling space and suppose that there are $T,T'\in\Omega_\Phi$, $r>0$, and $u\in\mathbb{S}^{d-1}$ so that $T-x\sim_sT'-x$ for all $x\in B_r(0)$ and $T-ru\nsim_sT'-ru$. Then there
is a pair $(\bar{T},\bar{T}')\in\mathcal{BP}(\Phi)$ with $\{v\in\mathbb{S}^{d-1}:\langle v,-u\rangle\,>0\}\subset S(\bar{T},\bar{T}')$.
\end{lemma}
\begin{proof} For $n\in\N$, let $T_n=\Phi^n(T-ru)$ and $T'_n=\Phi^n(T'-ru)$. By Lemma \ref{existence of R}, there is an $R$ so that $B_0[T_n-y]=B_0[T'_n-y]$ for all $y\in B_{\lambda^nr-R}(\lambda^nru)$ for each $n$. Let $\{n_i\}$ be a subsequence so that $T_{n_i}\to\tilde{T}$ and $T'_{n_i}\to\tilde{T'}$. Since $B_{R+1}[T_n]$ and $B_{R+1}[T'_n]$ share a tile, finite local complexity insures that there are only finitely many pairs of patches  $(B_{R+1}[T_n],B_{R+1}[T'_n])$, up to translation. It follows that from $T_{n_i}\nsim_sT'_{n_i}$, that $\tilde{T}\ne\tilde{T}'$. Now $\tilde{T}+(R+1)u\ne\tilde{T}'+(R+1)u$ and $B_0[\tilde{T}+(R+1)u-y]=B_0[\tilde{T}'+(R+1)u-y]$ for all $y\in\R^d$ with $\langle y,-u\rangle\,>0$. The desired $(\bar{T},\bar{T}')$ exists by Proposition \ref{branch points}.
\end{proof}

Let $\Phi$ be a substitution with branch pairs $\mathcal{BP}=\mathcal{BP}(\Phi)$ and periodic branch pairs $\mathcal{PBP}=\mathcal{PBP}(\Phi)$. Given $(T,T')\in\mathcal{BP}$, let $\tilde{\partial} S(T,T'):=\{v\in\partial S(T,T'):\exists t_i\to\infty$ such that $T-t_iv\nsim_sT'-t_iv\}$. The {\em branch locus} of $\Omega_{\Phi}$ is the set
$$
\mathcal{BL}(\Phi):=\cup_{(T,T')\in\mathcal{BP}}cl(\cup_{t\ge0,v\in \tilde{\partial} S(T,T')}\{T-tv\})
$$
and the {\em periodic branch locus} is the set
$$
\mathcal{PBL}(\Phi):=\cup_{(T,T')\in\mathcal{PBP}}cl(\cup_{t\ge0,v\in \tilde{\partial} S(T,T')}\{T-tv\}).
$$
It follows from Propositions \ref{agree in half-space} and \ref{branch points} that $\mathcal{BL}(\Phi)\ne\emptyset$ for a non-periodic, FLC, self-similar substitution tiling space $\Omega_{\Phi}$.

\begin{lemma}\label{branch pairs closed} If $\Omega=\Omega_{\Phi}$ is a $d$-dimensional, non-periodic, FLC, self-similar substitution tiling space, then $\mathcal{BP}(\Phi)$ is closed in $\Omega\times\Omega$.
If $(T_n,T'_n)\in\mathcal{BP}(\Phi)$ and $u_n\in\mathbb{S}^{d-1}$, $n\in\N$, are such that $\{v\in\mathbb{S}^{d-1}:\langle v,u_n\rangle\,>0\}\subset S(T_n,T'_n)$; $u$ is an accumulation point of $\{u_n:n\in\N\}$; and
$(T_n,T'_n)\to (T,T')$; then $\{v\in\mathbb{S}^{d-1}:\langle v,u\rangle\,>0\}\subset S(T,T')$.
\end{lemma}
\begin{proof} Suppose $(T_n,T'_n)\in\mathcal{BP}(\Phi)$ with $(T_n,T'_n)\to (T,T')\in\Omega\times\Omega$. Let $u_n\in\mathbb{S}^{d-1}$ be such that $\{v\in\mathbb{S}^{d-1}:\langle v,u_n\rangle\,>0\}\subset S(T_n,T'_n)$. By passing to a subsequence, we may assume that $u_n\to u$ as $n\to\infty$. As in the proof of Proposition \ref{branch points}, there are $x_n\to 0$ so that $(B_1[T_n],B_1[T'_n])=(B_1[T-x_n],B_1[T'-x_n])$ for all sufficiently large $n$. If $B_1[T]$ and $B_1[T']$ are stably related on $\{0\}$ then they are stably related on $B_{\epsilon}(0)$ for some $\epsilon>0$ (Lemma \ref{stable mfd}). But then $B_1[T_n]$ would be stably related to $B_1[T'_n]$ for $n$ large enough so that $|x_n|<\epsilon$, contrdicting $T_n\nsim_s T'_n$. Thus $T\nsim_s T'$.

Let $z\in\R^n$ be such that $\langle z,u\rangle\,>0$ and let $r=|z|+1$. For sufficiently large $n$, $(B_r[T_n],B_r[T'_n])=(B_r[T-x_n],B_r[T'-x_n])$. Then $B_r[T]$ is stably related to $B_r[T']$ on $\{z\}$ if and only if $B_r[T_n]$
is stably related to $B_r[T'_n]$ on $\{z-x_n\}$ (for large $n$). Since $u_n\to u$, $x_n\to 0$, and $\langle z,u\rangle\,>0$, $\langle z-x_n,u_n\rangle\,>0$ for sufficiently large $n$. For such $n$, $(z-x_n)/|z-x_n|\in S(T_n,T'_n)$ so
$T_n-(z-x_n)\sim_sT'-(z-x_n)$; hence  $B_r[T_n]$ is stably related to $B_r[T'_n]$ on $\{z-x_n\}$, $B_r[T]$ is stably related to $B_r[T']$ on $\{z\}$, and $T-z\sim_sT'-z$. So $(T,T')\in\mathcal{BP}(\Phi)$ with
$\{v\in\mathbb{S}^{d-1}:\langle v,u\rangle\,>0\}\subset S(T,T')$.
\end{proof}

\begin{theorem}\label{topological invariance 1} Suppose that $\Omega_{\Phi}$ and $\Omega_{\Psi}$ are homeomorphic $d$-dimensional, non-periodic, FLC, self-similar substitution tiling spaces and suppose that there are distinct tilings $T_0,T_0'\in\Omega_{\Phi}$ for which there is $S\subset \mathbb{S}^{d-1}$ that contains a closed hemisphere in its interior such that $T_0$ and $T_0'$ are uniformly asymptotic in directions $S$.
Then there is a homeomorphism from $\Omega_{\Phi}$ onto $\Omega_{\Psi}$ that takes $\mathcal{BL}(\Phi)$ onto $\mathcal{BL}(\Psi)$ and $\mathcal{PBL}(\Phi)$ onto $\mathcal{PBL}(\Psi)$.
\end{theorem}
\begin{proof} By Theorem 1.1 of \cite{Kwapisz}, there is a homeomorphism $h_0:\Omega_{\Phi}\to\Omega_{\Psi}$ and a linear isomorphism $L$ of $\mathbb{R}^d$ with $h_0(T-x)=h_0(T)-Lx$ for all $T\in\Omega_{\Phi}$ and all $x\in\mathbb{R}^d$. Let $u\in\mathbb{S}^{d-1}$ and $\alpha<0$ be such that
$\{v\in\mathbb{S}^{d-1}:\langle v,u\rangle\,>\alpha\}\subset S$. By Lemma \ref{asymptotic directions} and Propositions \ref{periodic branch pairs exist} and \ref{branch points}, there is $x_0$ so that $T_0-x_0$ and $T'_0-x_0$ are $\Phi$-periodic. Since $T_0-x_0$ and $T'_0-x_0$ are uniformly asymptotic in directions
$S$, it follows (from the ``linear" nature of $h_0$)  that $h_0(T_0-x_0)$ and $h_0(T'_0-x_0)$ are uniformly asymptotic in directions
$S':=\{Lv/|Lv|:v\in S\}$. The set of directions $S'$ must also contain a closed hemisphere in its interior, so, again by Lemma \ref{asymptotic directions} and Propositions \ref{periodic branch pairs exist} and \ref{branch points}, there is $y_0$ so that $h_0(T_0-x_0)-y_0$ and $h_0(T'_0-x_0)-y_0$ are $\Psi$-periodic. Let $h_1:\Omega_{\Phi}\to\Omega_{\Psi}$ be defined by $h_1(T):=h(T-x_0)-y_0$.
Since the homeomorphism $h_1$ is now ``pinned down" in that it takes a particular $\Phi$-periodic point to a $\Psi$-periodic point,
Corollary 1.5 of \cite{Kwapisz} guarantees that there is a homeomorphism $h:\Omega_{\Phi}\to\Omega_{\Psi}$ together with positive integers $m,n$ so that $h\circ\Phi^m=\Psi^n\circ h$. Such an $h$ necessarily takes $\Phi$-periodic points to $\Psi$-periodic points and
$\Phi$-stable manifolds to $\Psi$-stable manifolds and it follows that $h(\mathcal{BL}(\Phi))=\mathcal{BL}(\Psi)$ and $h(\mathcal{PBL}(\Phi))=\mathcal{PBL}(\Psi)$
\end{proof}

The branch locus may be rather difficult to describe in general. But in dimension two (as in dimension one) we will see that there are only finitely many sectors $S(T,T')$ for a given $\Phi$, which considerably simplifies the situation. We turn to that case now.

\section{Two-dimensional tilings}\label{2-d}

All 2-dimensional substitution tiling spaces in this section will be assumed to be non-periodic, primitive, self-similar, and have finite translational local complexity. For the sake of convenience, we will also assume all tiles have convex polygonal support.

\begin{lemma}\label{all sectors open} If $\Omega=\Omega_{\Phi}$ is a  $2$-dimensional, non-periodic, FLC, self-similar substitution tiling space and $(T,T')\in\mathcal{BP}(\Phi)$ then $S(T,T')$ is open.
\end{lemma}
\begin{proof} Given $(T,T')\in\mathcal{BP}(\Phi)$, if $S(T,T')$ contains a closed half-space in its interior, then $(T,T')$ is periodic and $S(T,T')$ is open by Lemma \ref{sector open}. So we may suppose that $\partial S(T,T')=\{v,-v\}$ for some $v\in\mathbb{S}^1$. Let $u\in\mathbb{S}^1$ be such that $\{v\in\mathbb{S}^1:\langle v,u\rangle\,>0\}\subset S(T,T')$. For $n\in\Z$, let $(T_n,T'_n):=(\Phi^{-n}(T),\Phi^{-n}(T'))$. Choose a subsequence $n_i$ so that $(T_{n_i},T'_{n_i})\to (\bar{T},\bar{T}')$ as $i\to\infty$. Then $(\bar{T},\bar{T}')\in\mathcal{BP}(\Phi)$, $\{v\in\mathbb{S}^1:\langle v,u\rangle\,>0\}\subset S(\bar{T},\bar{T}')$, and there are $x_i\to0$ so that $(B_1[T_{n_i}],B_1[T'_{n_i}])=(B_1[\bar{T}-x_i],B_1[\bar{T}'-x_i])$ for sufficiently large $i$ (Lemmas \ref{branch pairs closed} and \ref{finitely many patches}).

Let $\mathbf{l}$ be the line $\mathbf{l}:=\{tv:t\in\R\}$ and let $D^+$ be the component of $B_1(0)\setminus \mathbf{l}$ that contains $1/2u$.
Then $B_1[\bar{T}]$ is stably related to $B_1[\bar{T}']$ on $D^+$ so it must be the case that $x_i\notin D^+$. Now, if $x_i\notin\mathbf{l}$, then
$\langle -x_i,u\rangle\,>0$ so $T_{n_i}\sim_sT'_{n_i}$. But then $\bar{T}\sim_s\bar{T}'$, which is not the case. Thus
$x_i\in\mathbf{l}$ for all large $i$. If $x_i=0$ for infinitely many $i$, then $(T,T')$ is $\Phi$-periodic by Lemma \ref{periodic tilings} and $S(T,T')$ is open by Lemma \ref{sector open}. Otherwise, there are $i_1,i_2,i_3$ with $x_{i_2}$ between $x_{1_i}$ and $x_{i_3}$ on $\mathbf{l}$, say $x_{i_1}=x_{i_2}-t_1v$ and $x_{i_3}=x_{i_2}+t_3v$ with $t_1,t_3>0$. Then $T-\lambda^{n_{i_2}}t_1v\nsim_sT'-\lambda^{n_{i_2}}t_1v$ and  $T-\lambda^{n_{i_2}}t_3(-v)\nsim_sT'-\lambda^{n_{i_2}}t_3(-v)$, so $v\notin S(T,T')$,
$-v\notin S(T,T')$, and $S(T,T')$ is open.
\end{proof}

$\mathbf{N.B.}$ The above proof in fact shows that if $\Phi$ is 2-dimensional, $(T,T')\in \mathcal{BP}(\Phi)$ and $v\in\partial S(T,T')$, then there are $t_n\to\infty$ or $t_n\to-\infty$  so that $T-t_nv\nsim_s T'-t_nv$.

\begin{theorem}\label{finitely many sectors} If $\Omega=\Omega(\Phi)$ is a 2-dimensional, non-periodic, FLC, self-similar substitution tiling space, then the collection $\{S(T,T'):(T,T')\in\mathcal{BP}(\Phi)\}$ is finite.
\end{theorem}
\begin{proof}
Suppose instead that there are $(T_n,T'_n)\in\mathcal{BP}(\Phi)$, $n\in\N$, so that $\partial S(T_n,T'_n)=\{v_n,w_n\}$ and $\{v_i,w_i\}\ne\{v_j,w_j\}$ for all $i\ne j$. Passing to a subsequence, we may assume that $v_n\to v$ and $w_n\to w$ as $n\to\infty$. Since $v_n\notin S(T_n,T'_n)$ (Lemma \ref{all sectors open})
there are $s_n\in\R^+$ so that $T_n-s_nv_n\nsim_s T'_n-s_nv_n$. Let $k_n\in\Z$ be so that $\lambda^{k_n}s_n\in [\lambda^{-2},\lambda^{-1})$. Replace $(T_n,T'_n)$ by $(\Phi^{k_n}(T_n),\Phi^{k_n}(T'_n))$
and $s_n$ by $\lambda^{k_n}s_n$. We still have $\partial S(T_n,T'_n)=\{v_n,w_n\}$ since $\Phi$ is self-similar. Let $u_n\in \mathbb{S}^1$ be such that $\langle u_n,v_n\rangle=0$ and $\{v\in\mathbb{S}^1:\langle v,u\rangle\,>0\}\subset S(T_n,T'_n)$. By passing to a subsequence, we may assume that $(T_n,T'_n)\to (T,T')\in\Omega\times\Omega$, $s_n\to s\in [\lambda^{-2},\lambda^{-1}]$, and $u_n\to u\in\mathbb{S}^1$. From Lemma \ref{branch pairs closed} we have $(T,T')\in\mathcal{BP}(\Phi)$, $\{v\in\mathcal{S}^1:\langle v,u\rangle\,>0\}\subset
S(T,T')$, and $T-sv\nsim_s T'-sv$ (the latter because $(T_n-s_nv_n,T'_n-s_nv_n)\in\mathcal{BP}(\Phi)$ and $(T_n-s_nv_n,T'_n-s_nv_n)\to (T-sv,T'-sv)$). By Lemma \ref{finitely many patches} there is an $N$ and there are $x_n\to 0$ so that $(B_1[T_n],B_1[T'_n])=(B_1[T-x_n],B_1[T'-x_n])$ for all $n\ge N$. By taking $N$ large enough we can assure that $|x_n|<\epsilon$ and $|x_n+s_nv_n-sv|<\epsilon$ for all $n\ge N$, where $\epsilon>0$ is small enough so that the $\epsilon$ balls centered at $sv$ and $0$ are disjoint and contained in $B_1(0)$. Let $\mathbf{l}$ and $\mathbf{l}_n$ be the lines $\{tv:t\in\R\}$ and $\{tv_n+x_n:t\in\R\}$, resp.. Let $D^{\pm}$ and $D_n^{\pm}$ be the components of $B_1(0)\setminus \mathbf{l}$ and $B_1(0)\setminus \mathbf{l}_n$, resp., with $1/2u\in D^+$ and $1/2u_n+x_n\in D_n^+$.
Then $B_1[T]$ is stably related to $B_1[T']$ on $D^+$. Pick $n\ge N$. Since $(B_1[T_n],B_1[T'_n])=(B_1[T-x_n],B_1[T'-x_n])$ and $T_n\nsim_s T'_n$, $B_1[T]$ and $B_1[T']$ are not stably related on $\{x_n\}$; that is $x_n\notin D^+$. Similarly, $x_n+s_nv_n\notin D^+$. Suppose that $v_n\ne v$. Then the lines $\mathbf{l}_n$ and $\mathbf{l}$ intersect in a single point $\{p=tv\}$. If $t\ge s$, then $0\in\ D_n^+$; if $t\le 0$, then $sv\in D_n^+$; and if $0<t<s$ then exactly one of $0,sv$ is in $D_n^+$. So if $v_n\ne v$, $\{0,sv\}\cap D_n^+\ne\emptyset$. If $0\in D_n^+$ then $\langle u_n,-x_n\rangle\,>0$ so $T_n+x_n\sim_s T'+x_n$. But then $T\sim_s T'$, which is not the case. Similarly, if $sv\in D_n^+$ then $\langle u_n,sv-x_n\rangle\,>0$ so that $T_n-(sv-x_n)\sim_s T'-(sv-x_n)$. But then $T-sv\sim_s T'-sv$, which is also not the case. Thus it must be that $v_n=v$ for all $n\ge N$.

Now, for each $n\ge N$ there are $t_n>0$ so that $T_n-t_nw_n\nsim_sT'_n-t_nw_n$. Let $m_n\in\Z$ be such that $\lambda^{m_n}t_n\in [\lambda^{-2},\lambda^{-1})$: replace $(T_n,T'_n)$ by $(\Phi^{m_n}(T_n),\Phi^{m_n}(T'_n))$ and $t_n$ by $\lambda^{m_n}t_n$. Let $\bar{u}_n\in\mathbb{S}^1$ be such that $\{v\in\mathbb{S}^1:\langle v,\bar{u}_n\rangle\,>0\}\subset S(T_n,T'_n )$ for $n\ge N$. Passing to a subsequence, we may assume that $(T_n,T'_n)\to (\bar{T},\bar{T}')\in\Omega\times\Omega$, $t_n\to\bar{t}$, and $\bar{u}_n\to \bar{u}$ as $n\to\infty$. Note that for the new $T_n,T'_n$, it's still the case that $\partial S(T_n,T'_n)=\{v,w_n\}$ for $n\ge N$, with $w_i\ne w_j$ for $i\ne j$. Proceeding just as above, we may conclude that $w_n\equiv w$ for all sufficiently large $n$. That is, there could only have been finitely many asymptotic sectors to begin with.
\end{proof}

We will say that $\mathcal{BP}(\Phi)$ is {\em non-collapsing} if, whenever $(T,T')\in\mathcal{BP}(\Phi)$, $v\in\partial S(T,T')$, and $0\le t_0< t_1$ are such that $T-t_0v\nsim_sT'-t_0v$ and $T-t_1v\nsim_sT'-t_1v$, then $T-tv\nsim_sT'-tv$ for all $t\in[t_0,t_1]$. It is clear from Lemma \ref{branch pairs closed}
that the branch locus is non-collapsing if and only if the branch locus coincides with the collection of branch pairs.

If $(T,T')$ is a branch pair in the 2-dimensional substitution tiling space $\Omega$, we'll call $(T,T')$
\begin{itemize}
\item  an {\em isolated pair} if $S(T,T')=\mathbb{S}^1$;
\item a {\em corner pair} if $S(T,T')$ is proper in $\mathbb{S}^1$ but contains a closed semicircle in its interior;
\item a {\em line pair} if $S(T,T')$ is an open semicircle.
\end{itemize}
Every branch pair is of one of the above types. Furthermore, there are only finitely many isolated and corner pairs
and they are all periodic (Lemma \ref{all sectors open} and Theorem \ref{finitely many sectors}).

\begin{proposition}\label{PBL=BL} If $\mathcal{BP}$ is non-collapsing then $\mathcal{PBL}(\Phi)=\mathcal{BL}(\Phi)$.
\end{proposition}
\begin{proof} Pick $(T,T')\in\mathcal{BP}(\Phi)$. Suppose there are $v\in\partial S(T,T')$ and $t>0$ with $T-tv\sim_sT'-tv$. Let $t_0:=inf\{t>0:T-tv\sim_sT'-tv\}$. Then $(T-t_0v,T'-t_0v)$ is either an isolated or a corner pair and hence is periodic. If $t_0\ne0$, then $\tilde{\partial} S(T,T')=\{-v\}\subset \tilde{\partial} S(T-t_0v,T'-t_0v)$ and $\cup_{t\ge0}\{T-t(-v)\}\subset\cup_{t\ge0}\{(T-t_0v)-t(-v)\}$. If there are not such $v$ and $t$, then either $(T,T')$ is an isolated pair, and hence in $\mathcal{PBP}$, or $(T,T')$ is a line pair
with $\tilde{\partial} S(T,T')=\{\pm v\}$ for some $v\in\mathbb{S}^1$. In the latter case, and using Lemma
\ref{finitely many patches}, there is a $k\in\N$ and a sequence $n_i\to\infty$ so that all of the pairs of patches $(B_1[\Phi^{-n_i}(T)],B_1[\Phi^{-n_i}(T')])$ and $(B_1[\Phi^{-n_i-k}(T)],B_1[\Phi^{-n_i-k}(T')])$ are
translates of the same pair of patches $(P,P')$. Moreover, we may take $n_i\equiv n_j$ mod($k$), and,
without loss of generality $n_i\equiv0$ mod($k$). As in the proof of Proposition \ref{periodic branch pairs exist}, the vectors translating the pairs of patches $(B_1[\Phi^{-n_i}(T)],B_1[\Phi^{-n_i}(T')])$ to the pairs of patches $(B_1[\Phi^{-n_i-k}(T)],B_1[\Phi^{-n_i-k}(T')])$ are all parallel with $v$. Passing to a subsequence, we may assume that $(B_1[\Phi^{-n_i-k}(T)],B_1[\Phi^{-n_i-k}(T')])\to (Q,Q')$. We have that
$\Phi^k(Q)\supset Q-sv$ and $\Phi^k(Q')\supset Q'-sv$ for some $s$. Then $\Phi^k(Q-(\frac{s}{1-\lambda^k})v)\supset Q-(\frac{s}{1-\lambda^k})v$ and  $\Phi^k(Q'-(\frac{s}{1-\lambda^k})v)\supset Q'-(\frac{s}{1-\lambda^k})v$. Let $\bar{T}:=\cup_{n\ge0}\Phi^{nk}(Q-(\frac{s}{1-\lambda^k})v)$ and
$\bar{T}':=\cup_{n\ge0}\Phi^{nk}(Q'-(\frac{s}{1-\lambda^k})v)$. Then $(\bar{T},\bar{T}')$ is $\Phi$-periodic
of period $k$ and is a line branch pair with $\tilde{\partial} S(\bar{T},\bar{T}')=\{\pm v\}$ by the non-collapsing assumption. Moreover, there are bounded $s_i$ so that $B_1[\Phi^{-n_i}(T)]=B_1[\bar{T}-s_iv]$. Let $n_i=m_ik$. Then, for any $t\in\R$, $B_{\lambda^{m_ik}}[T-tv]=B_{\lambda^{m_ik}}[\bar{T}-(\lambda^{m_ik}s_i-t)v]$ for all $i$ so that $\bar{T}-(\lambda^{m_ik}s_i-t)v\to T-tv$ as $i\to\infty$. Thus $T-tv\in\mathcal{PBL}$ so that $\mathcal{BL}=\mathcal{PBL}$.
\end{proof}

\smallskip
\noindent\textbf{Example 1: The half-hex.}
\smallskip

This is a self-similar Pisot substitution tiling space with $\lambda=2$ and substitution rule shown in Figure \ref{fig:half hex substitution}.  There are six prototiles, consisting of a half hexagon tile and its rotations through multiples of $\pi/6$.

\begin{figure}[htb]
\begin{center}
\includegraphics[height=.8in, width=3in]{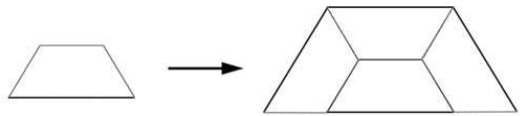}
\end{center}
\caption{Half-hex substitution rule.}
\label{fig:half hex substitution}
\end{figure}

 Six branch pairs are created from three distinct tilings, $T_1$, $T_2$, and $T_3$, with patches  at the origin shown in Figure \ref{fig:half hex patterns}.  These tilings are fixed by the substitution and are identical off the support of the pictured patches; hence $S(T_i, T_j) = \mathbb{S}^1, i, j \in \{ 1,2,3\}$, $i\not = j$. That is, these are isolated pairs. There are no other branch pairs so the branch locus is $\mathcal{BL} = \{T_1, T_2, T_3\}$.

\begin{figure}[htb]
\begin{center}
\includegraphics[height=1in, width=3in]{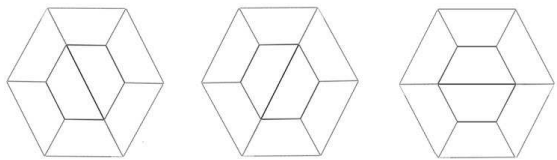}
\end{center}
\caption{The 1-patches of the three elements of the half-hex branch locus. }
\label{fig:half hex patterns}
\end{figure}

\smallskip
\noindent\textbf{Example 2: The chair.}
\smallskip

This is also a self-similar Pisot substitution tiling space with $\lambda=2$. There are four square prototiles for the chair substitution $\Phi_C$. (We will describe the ``square chair" - see \cite{Sa} for it's equivalence with the more familiar substitution with chair-shaped tiles.) Let $\rho_0$ be the unit square tile with the origin at the center and marked with a northeast pointing arrow, and let $\rho_i:=r^i(\rho_0)$, where
$r$ denotes rotation through $\pi/2$. By this we mean, for example, that $\rho_1$ has the same support as $\rho_0$ but is marked with a northwest pointing arrow. The substitution on $\rho_0$ inflates by a factor of $2$ and fills in the $2\times2$ square with translates of the $\rho_i$ as pictured.

\smallskip
$$\neabox \to \begin{matrix}
\seabox \neabox\cr
\neabox \nwabox \end{matrix}$$
\smallskip

Let $\Phi_C(\rho_i):=r^i(\Phi_C(\rho_0))$. The tilings $T_i:=\cup_{n\ge1}\Phi_C^n(\rho_i)$ are fixed by
$\Phi_C$ and are members of corner pairs: $S(T_i,T_{i+1})=r^i(\{v\in\mathbb{S}^1:\langle v,(0,-1)\rangle>-\sqrt{2}/2\})$, with the subscripts on the $T$'s taken mod(4). These are all the corner pairs, there are no isolated pairs, and all members of line pairs have zero-patch equal to: a translate of $\rho_0$ or $\rho_2$ in a northeast-southwest direction; a translate of $\rho_1$ or $\rho_3$ in a northwest-southeast direction; or a collection of four prototiles with all arrows pointing towards, or away from, the origin. There are five zero-patches of the latter form: $B_0[T_i],i=0,\ldots,3$, and the patch $P$ pictured below.

\smallskip
$$P=\begin{matrix}
\nwabox \neabox\cr
\swabox \seabox \end{matrix}$$
\smallskip

Let $T^*:=\cup_{n\in\N}\Phi_C^n(P)$ be the tiling, fixed under $\Phi_C$ with $B_0[T^*]=P$. The branch locus of $\Phi_C$ then consists of the disjoint union of four dyadic solenoids,
$S_i:=cl(\{T_i-tr^i((\sqrt{2}/2,\sqrt{2}/2)):t\in\R\}),i=0,\ldots,3$, together with four rays, $R_i:=
\{T^*-tr^i((\sqrt{2}/2,\sqrt{2}/2)):t\ge0\}, i=0,\ldots,3$, joined at their common endpoint $T^*$. Each ray $R_i$ winds densely on $S_i$. $\mathcal{BP}(\Phi_C)$ is non-collapsing, so $\mathcal{PBL}(\Phi_C)=
\mathcal{BL}(\Phi_C)$.

The branch locus for the chair tiling can be described as an inverse limit of an expanding Markov map
(see Theorem \ref{structure of branch locus} below) as follows. Let $X:=\T\cup [1,3]$ be the ``circle with sticker"  in the complex plane consisting of the union of the unit circle $\T$ and the interval on the real axis from 1 to 3, and let $f_0:X\to X$ by:
\[f_0(z)=\left\{ \begin{array}{ll}
   z^2 &  \mbox{if $z\in\T$},\\
   exp(2\pi zi) &  \mbox{if $z\in [1,2]$},\\
   2z-3 & \mbox{if $z\in [2,3].$}
 \end{array} \right.\]
 Now let $K:=\cup_{i=0,\ldots,3}r^i(X-3)$ be the union of four copies of $X$, joined at their endpoints,
 and let $f:K\to K$ be given by $f(z):=r^i(f_0(r^{-i}(z)+3)-3)$ for $z\in r^i(X-3)$. Then $\inv f$ is homeomorphic with $\mathcal{BL}(\Phi_C)$ by a homeomorphism that conjugates $\hat{f}$ with $\Phi_C|_{\mathcal{BL}(\Phi_C)}$.
\smallskip

\smallskip
\noindent\textbf{Example 3: The octagonal tiling.}
\smallskip

The (undecorated) octagonal tiling space is a self-similar Pisot substitution tiling space with $\lambda=1+\sqrt{2}$. There are 20 prototiles: four unmarked rhombi, $\rho_i=r^i(\rho_0)$, $i=0,\ldots3$, with $r^4(\rho_0)=\rho_0$, $r$ being rotation through $\pi/4$,  with the origin at their centers; and sixteen marked isosceles right triangles $\tau_i=r^i(\tau_0),\tau'_i=r^i(\tau'_0)$,
$i=0,\ldots,7$, with $\tau_i$ and $\tau'_i$ having the same support but bearing different marks. The origin is at the midpoint of the hypotenuse of the triangular prototiles. The octagonal substitution, $\Phi_O$, is described in the following figure.

\begin{figure}[htb]
\begin{center}
\includegraphics[height=3in, width=3in]{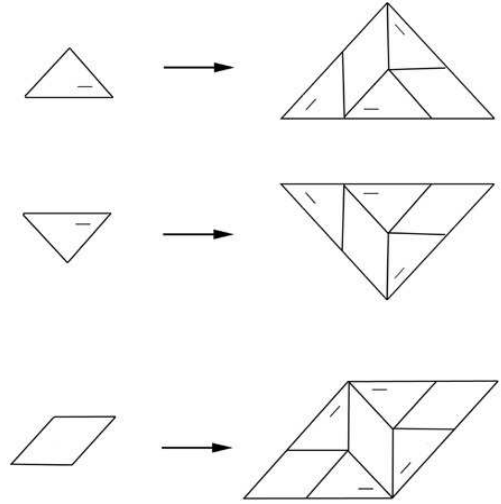}
\end{center}
\caption{Octagonal substitution rule.}
\label{fig:Oct substitution}
\end{figure}

If $\rho'_0$ denotes the translate of $\rho_0$ with the origin slightly off-center (see Figure \ref{fig:Oct one patches}), then $\rho'_0\in\Phi_O^2(\rho'_0)$ so that the tilings $T_{1,i}:=\cup_{n\in\N}\Phi_O^{2n}(r^i(\rho'_0))$, $i=0,\ldots,7$, are fixed by $\Phi_O^2$. Also, $\tau_i\in\Phi_O^2(\tau_i)$ and $\tau'_i\in\Phi_O^2(\tau'_i)$, so the tilings $T_{2,i}:=\cup_{n\in\N}\Phi_O^{2n}(\{\tau_i,\tau'_{i+4}\})$, $i=0,\ldots,7$ and $i+4$ taken mod(8), are fixed by $\Phi_O^2$. 
Note that $\Phi_O(T_{j,i})=T_{j,i+4}$, for $j=1,2$ and $i=0,\ldots,7$, with subscripts taken mod(8).

\begin{figure}[htb]
\begin{center}
\includegraphics[height=1.5in, width=3in]{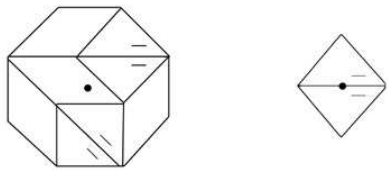}
\end{center}
\caption{1-patches of the tilings $T_{1,0}$ and $T_{2,0}$. Black dots indicate the placement of the origin. }
\label{fig:Oct one patches}
\end{figure}

The tilings $T_{1,i}$ and $T_{1,i+3}$ are members of a branch line pair, as are $T_{1,i}$ and $T_{1,i+5}$
(subscripts taken mod(8)), and the corresponding sectors have boundaries: $\partial S(T_{1,i},T_{1,i+3})=\{\pm r^i((\sqrt{2}/2,\sqrt{2}/2))\}$; and $\partial S(T_{1,i},T_{1,i+5})=\{\pm r^i((0,1))\}$. Let $\Omega_i:=
cl(\{T_{1,i}-tr^i((\sqrt{2}/2,\sqrt{2}/2)):t\in\R\})$. $\mathcal{BP}(\Phi_O)$ is non-collapsing, and the branch locus of the octagonal tiling is $\mathcal{BL}(\Phi_O)=\mathcal{PBL}(\Phi_O)=\cup_{i=0}^7\Omega_i$ with the $\Omega_i$ intersecting as follows: $$\Omega_0\cap\Omega_j=
\left\{\begin{array}{cc} \{T_{1,0}\}, & j=1\\ \{T_{2,0}\}, & j=2\\ \emptyset, & j=3\\ \emptyset, & j=4\\ \emptyset, & j=5\\ \{T_{2,6}\}, & j=6\\ \{T_{1,7}\}, & j=7.\end{array} \right. $$
Note that $r$ determines a homeomorphism of $\Omega_O$ by $r(\{\sigma_i+v_i\}):=\{r(\sigma_i)+r(v_i)\}$ for $v_i\in\R^2$ and prototiles $\sigma_i$. Then $\Omega_i=r^i(\Omega_0)$ and the intersections of 
$\Omega_i$ with the various $\Omega_j$ are obtained by applying $r_i$ to the above. $\Omega_0$ is itself (homeomorphic with) a one-dimensional substitution tiling space which can be described as follows. Let $X$ be a wedge of three circles, labeled by $1,2$ and $3$, and let $f:X\to X$ be the Markov map that maps these circles according to the pattern of the substitution $\phi$:
$$1\mapsto 23131$$
$$2\mapsto 23231$$
$$3\mapsto2323131.$$
Then $\Phi_O^2|_{\Omega_0}$ is conjugate with the shift $\hat{f}$ on $\Omega_{\phi}\simeq\inv f\simeq\Omega_0$. To describe all of $\mathcal{BL}(\Phi_O)$, let $x_0,x_1,x_2$, and $x_3$ be the fixed points of $\hat{f}$ corresponding to
the last occurrence of 3 in $\phi(3)$, the first occurrence of 3 in $\phi(3)$, the first occurrence of 1 in $\phi(1)$, and the second occurrence of 2 in $\phi(2)$, resp. (These points correspond to $T_{1,0},T_{1,7},T_{2,0}$, and $T_{2,6}$, resp.) Let $B:=(\cup_{i=0}^7\inv{f}\times \{i\})/\sim$, in which $\sim$ identifies points as follows: $(x_0,i)\sim(x_1,i+1),(x_1,i)\sim(x_0,i+7),(x_2,i)\sim(x_3,i+2)$, and $(x_3,i)\sim(x_2,i+6)$. $\Phi_O^2|_{\mathcal{BL}(\Phi_O)}$ is then conjugate with $F:B\to B$, where $F([(x,i)]):=[(\hat{f}(x),i)]$.

\begin{figure}[htb]
\begin{center}
\includegraphics[height=2in, width=2in]{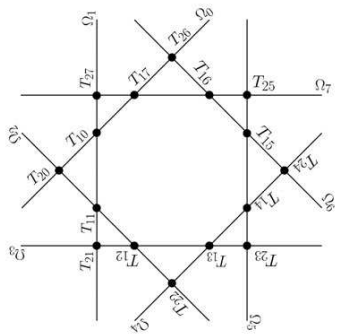}
\end{center}
\caption{Schematic of the branch locus of the octagonal tiling space.}
\label{fig:Oct branch locus}
\end{figure}

The octagon tilings can also be obtained from a cut-and-project scheme (see \cite{FHK}) in which the ``window" is an octagon. In this presentation, the ambiguity introduced by each of the eight  boundary segments of the octagon is responsible for each of the eight pieces $\Omega_i$ of the branch locus.

\bigskip

The following Theorem guarantees that, unless the branch locus consists of just a finite number of points (as for the half-hex), the branch locus is rather complicated.

\begin{theorem}\label{two sectors} If $\Omega$ is a 2-dimensional, non-periodic, FLC,  self-similar substitution tiling space
that has no isolated branch pairs, then there are branch pairs $(T_1,T'_1),(T_2,T'_2)\in\Omega\times\Omega$
with $\partial S(T_1,T'_1)\ne \partial S(T_2,T'_2)$. Moreover, if $\Omega$ also has no corner branch pairs, then all line branch pairs  $(T,T')$ are non-collapsing in the sense that $T-tv\nsim_sT'-tv$ for all $t\in\R,v\in\partial S(T,T')$.
\end{theorem}
{\em Sketch of proof:} We will see rather easily that if $\Omega$ has a corner pair then $\Omega$ also has a line pair. With a bit more work we will show that if $\Omega$ has only line pairs with parallel directions, then $\Omega$ has a nonzero translational period. The detailed proof is deferred to the next section.
\\
\\
We'll say that a continuous map $f:K\to K$ on a finite 1-dimensional simplicial complex $K$ with vertices $V$ is an {\em expanding Markov} map provided $f(V)\subset V$, $f^{-1}(V)$ is finite, and $f$ is one-to-one and stretches length by a factor of at least $\lambda$, for some $\lambda>1$, on each component of $K\setminus (V\cup f^{-1}(V))$. Such maps are quite similar to the Williams presentations of 1-dimensional hyperbolic attractors (\cite{W}). In particular, at all but finitely many points, the inverse limit space $\Lim f$ is homeomorphic with the product of a 0-dimensional set with an arc.

\begin{theorem}\label{structure of branch locus} If $\Omega_{\Phi}$ is a 2-dimensional, non-periodic, FLC, self-similar substitution tiling space and either $\mathcal{BP}(\Phi)$ is non-collapsing or $\Phi$ is Pisot, then there is a finite 1-dimensional simplicial complex $K$ and an expanding Markov map $f:K\to K$ so that $\Phi|_{\mathcal{PBL}(\Phi)}$ is topologically conjugate with $\hat{f}$ on
$\Lim f$.
\end{theorem}

{\em Sketch of proof:} Let $X$ be the 1-collared Anderson-Putnam complex for $\Omega$ with $\pi:\Omega\to X$ and $F:X\to X$ so that $\pi\circ\Phi=F\circ\pi$. Then $\hat{\pi}:\Omega\to \hat{X}:=\inv F$
is a homeomorphism that conjugates $\Phi$ with $\hat{F}$ (see \cite{AP}). We will prove in the next section that, under the hypotheses of Theorem \ref{structure of branch locus}, the collection $\{(\tau,l)\}$
consisting of all pairs with $\tau\in T$ and $\tau\cap l\ne\emptyset$, where either $(T,T')$ is a periodic line pair for which $\tilde{\partial} S(T,T')=\partial S(T,T')=\{\pm v\}$ and $l=\{tv:t\in\R\}$, or $(T,T')$ is a corner pair with $v\in\tilde{\partial} S(T,T')$ and $l=\{tv:t\ge0\}$, is finite, up to translation. From this it will follow that the set $K:=\pi(\mathcal{PBL}(\Phi))\subset X$ is a finite 1-dimensional simplicial complex invariant under $F$. Then $f=F|_K$ is expanding Markov and $\Phi|_{\mathcal{PBL}(\Phi)}$ is topologically conjugate with $\hat{f}$ on $\Lim f$. A complete proof appears in the next section.

\begin{theorem}\label{topological invariance 2} Suppose that $\Omega$ and $\Omega'$ are homeomorphic 2-dimensional, non-periodic, FLC, self-similar substitution tiling spaces and at least one of the following conditions is met:
\begin{itemize}
\item $\Omega$ has an isolated or a corner branch pair;
\item $\Omega$ has line pairs $(T_1,T_2)$ and $(T_2,T_3)$ with $\partial S(T_1,T_2)\ne \partial S(T_2,T_3)$;
\item $\Omega$ is Pisot.
\end{itemize}
Then there is a homeomorphism from $\Omega$ to $\Omega'$ that takes branch pairs to branch pairs, preserving type, and that restricts to a homeomorphism of the branch loci and of the periodic branch loci.
\end{theorem}

{\em Sketch of proof:} As was the case for the proof of Theorem \ref{topological invariance 1}, the rigidity results of \cite{Kwapisz} will provide the key
to the proof of Theorem \ref{topological invariance 2}. First, we may replace an arbitrary homeomorphism of $\Omega$ and $\Omega'$ by a `linear' homeomorphism that preserves asymptoticity. To further improve the linear homeomorphism to a conjugacy (between some powers of
the substitution maps) it is necessary to show that the linear homeomorphism can be made to be `pointed', that is, to take some self-similar (i.e., periodic under substitution) tiling to a self-similar tiling.
This is accomplished by establishing the existence of tilings with such special asymptotic properties, that any tiling with these properties must be self-similar (at least up to translation). This existence is established by fiat in the first two of the alternative hypotheses of Theorem \ref{topological invariance 2}.
From Theorem \ref{two sectors} it follows that if $\Omega$ has neither isolated nor corner branch pairs, then $\Omega$ has non-parallel line branch pairs. Intersection of the corresponding lines would yield the second of the alternative hypotheses. In case $\Omega$ is Pisot we will see that these lines have a
virtual intersection when viewed in the maximal equicontinuous factor of $\Omega$. This virtual intersection will prove sufficient to pin down a linear homeomorphism of $\Omega$ and $\Omega'$ to a pointed homeomorphism. Once powers of the substitutions on $\Omega$ and $\Omega'$ are conjugated, Theorem \ref{topological invariance 2} will follow easily.

\section{Proofs of the main theorems}\label{proofs}

Throughout this section $\Omega_{\Phi}$ will always be a 2-dimensional non-periodic, minimal, self-similar tiling space with translationally finite local complexity.

Given $\epsilon>0$ and $T\in\Omega_{\Phi}$, a line $l$ with direction vector $v\in\mathbb{S}^1$ will be called an
{\em $\epsilon$-characteristic} for $T$ provided there is a $w\in\R^2$ and a $v^{\perp}\in\mathbb{S}^1$,
with $<v,v^{\perp}>=0$, so that for all $y\in l$ and $s\in (0,\epsilon]$: $T-y\nsim_s T-w-y$; and $T-y-sv^{\perp}\sim_s T-w-y-sv^{\perp}$.

Let $[(*,*)]$ denote the translation equivalence class of $(*,*)$.

\begin{lemma}\label{finitely many cuts}
Fix $\epsilon>0$ and $T\in\Omega_{\Phi}$. Then $\{[(\tau,\tau\cap l)]: l$ is an $\epsilon$-characteristic for $T$ and $\tau\in T\}$ is finite.
\end{lemma}
\begin{proof}
For each $\epsilon$-characteristic  $l$ there is a line pair $(T',T'')\in\mathcal{BP}(\Phi)$ with $\partial S(T',T'')=\{\pm v\}$ (let $(T',T'')$ be any subsequential limit of $(\Phi^k(T-y),\Phi^k(T-y-w))$, $k\in\Z, n\in\N,y\in l$), so there are only finitely many directions of $\epsilon$-characteristics by Theorem \ref{finitely many sectors}. It follows that there are only finitely many pairs $[(\tau,\tau\cap l)]$ with $\mathring{\tau}\cap l =\emptyset$ (tiles are connected). Thus, if there are infinitely many distinct $[(\tau,\tau\cap l)]$, there is a prototile $\rho$,  vectors $x_n$ with $\rho-x_n\in T$, and $\epsilon$-characteristics $l_n$ with $v^{\perp}(l_n)\equiv v^{\perp}$ so that $[(\rho+x_n, (\rho+x_n)\cap l_n)]$ are all distinct and $(\mathring{\rho}+x_n)\cap l_n\ne\emptyset$ for all $n\in\N$. Pick $y_n\in (\mathring{\rho}+x_n)\cap l_n$ and let $w_n=w(l_n)$ be as in the definition of $\epsilon$-characteristic. We may assume that the $B_1[T-y_n]$ are all of the same type,
say $B_1[T-y_n]=B-z_n$, and that the $B_1[T-w_n-y_n]$ are all of the same type, say $B_1[T-w_n-y_n]=
B'-z_n'$, with $z_n\to 0$ and $z_n'\to 0$. Note that since the $[(\rho+x_n, (\rho+x_n)\cap l_n)]$ are all distinct, $\langle z_n-z_m,v^{\perp}\rangle\ne 0$ for $n\ne m$.

Now $B-z_n$ is stably related with $B'-z_n'$ on $\{x:|x|<\epsilon$ and $\langle x,v^{\perp}\rangle>0\}$. It is a consequence of Lemma \ref{finitely many patch pairs} that $\{z_n'-z_n\}$ is finite. Pick a large $n$ and a $k>0$ with $|z_{n+k}-z_n|<\epsilon$ and $z'_{n+k}-z_{n+k}=z'_n-z_n$. Suppose $\langle z_{n+k}-z_n,v^{\perp}\rangle<0$. We have
$B-z_{n+k}\nsim_sB'-z'_{n+k}$ and $B-z_{n+k}+(z_{n+k}-z_n)\sim_sB'-z'_{n+k}+(z_{n+k}-z_n)$. But
$z_{n+k}-z_n=z'_{n+k}-z'_n$, so we have  $B-z_n\sim_sB'-z'_n$, that is, $T-y_n\sim T-w_n-y_n$, contradicting the assumption that $l_n$ is an $\epsilon$-characteristic for $T$. A similar contradiction is reached
if $\langle z_{n+k}-z_n,v^{\perp}\rangle>0$, hence there can be only finitely many  $[(\tau,\tau\cap l)]$.
\end{proof}

It follows from the above that there is a $\delta=\delta(\epsilon)>0$ so that if $l$ and $l'$ are two distinct, parallel $\epsilon$-characteristics for $T$, then the distance between $l$ and $l'$ is at least $\delta$.

\begin{lemma}\label{strips}
Suppose that $T\in\Omega_{\Phi}$ is fixed by $\Phi$ and that $\{l_n\}_{n\in\Z}$ is a family of parallel lines so that $l_n$ separates $l_{n-1}$ from $l_{n+1}$ for all $n$ and $\R^2=\cup_{n\in\Z} S_n$ where $S_n$ is the
open strip between $l_{n-1}$ and $l_{n+1}$. Suppose further that $\lambda\{l_n\}_{n\in\Z}\subset \{l_n\}_{n\in\Z}$ and there are nonzero $w_n$ parallel to the $l_n$,
 so that $T-y\sim_sT-w_n-y$ for all $y\in S_n$. Then $\{w_n\}$ is infinite.
\end{lemma}
\begin{proof}
Suppose otherwise. Say $0\in S_0$. There is then $N$ large enough so that $\lambda^N\{l_{-1},l_1\}=\{l_{-n_1-1},l_{n_2+1}\}$
with  $\{w_n\}_{n\in\Z}=\{w_j\}_{j=-n_1\ldots,n_2}$. For $j\in\{-n_1,\ldots,n_2\}$ and $y\in S_j$ we have
$T-y\sim_sT-w_j-y$ and $T-y\sim_sT-\lambda^Nw_0$. Let $M_j:=\{t\in\R:T-y\sim_sT-tw_j-y$ for all $y\in S_j\}$. Then $M_j$ is a $\Z$-module. By Lemma \ref{no slip}, $M_j$ must have a smallest positive element, so $M_j$ is isomorphic with $\Z$. Thus, since $|\lambda^Nw_0|/|w_j|\in M_j$ for $j\in\{-n_1,\ldots,n_2\}$, there must be nonzero $k_j,m_j\in\Z$ with $k_jw_j=m_j\lambda^Nw_0$. Then $T-y\sim_sT-mw_0-y$ for all $y\in\R^2$ where $m=\lambda^N\prod_{j=-n_1}^{n_2}m_j$. It follows from Lemma \ref{T=T'} that $T-mw_0=T$, contradicting the non-periodicity of $\Omega_\Phi$.
\end{proof}

\begin{lemma} \label{corner}
Suppose that $T,T'\in\Omega_\Phi$, $u\in\mathbb{S}^1$, and $\delta>0$ are such that $T-y\sim_sT'-y$ for all $y$ with $0<\langle y,u\rangle<\delta$.
Suppose also that $T\nsim_sT'$ and there is $y_0$ with $\langle y_0,u\rangle=0$ so that $T-y_0\sim_sT'-y_0$. There is
then a pair $(\bar{T},\bar{T}')\in\mathcal{BP}(\Phi)$ and a vector $w$ not parallel with $u$ so that $\{v\in\mathbb{S}^1:\langle v,w\rangle \,>0\}\subset S(\bar{T},\bar{T}')$.
\end{lemma}
\begin{proof}
There is $\epsilon>0$ so that $T-y\sim_sT'-y$ for all $y\in cl(B_{\epsilon}(y_0))$. Consider the
curves $c_s:[-\epsilon,\epsilon]\to\R^2$ given by $c_s(t)=t(\frac{v}{|v|})+\{(\frac{\epsilon(1-s)+s|y_0|}{\epsilon^2})t^2-(1-s)\epsilon-s|y_0|\frac{y_0}{|y_0|}\}+y$. For each $s$, the image of $c_s$ is an arc of a parabola from $y_0-\epsilon\frac{v}{|v|}$ to $y_0+\epsilon\frac{v}{|v|}$ with vertex at $y_0-[(1-s)\epsilon+s|y_0|]\frac{y_0}{|y_0|}$. When $s=0$ this curve lies entirely in $cl(B_{\epsilon}(y_0))$ and when
$s=1$ the vertex is at $0$. Let $s_0$ be the infimum of those $s>0$ for which there is $t\in [-\epsilon,\epsilon]$ with $T-c_s(t)\nsim_sT'-c_s(t)$. Then there is $t_0\in (-\epsilon,0]$ with $T-c_{s_0}(t_o)\nsim_sT'-c_{s_0}(t_0)$. Let $w$ be the unit normal to the curve $c_{s_0}$ at $c_{s_0}(t_0)$ with
$\langle w,y_0\rangle>0$. For sufficiently small $r>0$ we have: $(T-c_{s_0}(t_0)-rw)-y\sim_s(T'-c_{s_0}(t_0)-rw)-y$ for all $y\in B_r(0))$ and  $(T-c_{s_0}(t_0)-rw)+rw\nsim_s(T'-c_{s_0}(t_0)-rw)+rw$. By Lemma \ref{branch points with direction} there is a branch pair $(\bar{T},\bar{T}')$ as desired.
\end{proof}

\noindent {\em Proof of Theorem \ref{two sectors}}. For the first part, we prove the equivalent statement: If $\Omega_\Phi$ is a 2-dimensional FLC self-similar substitution tiling space with connected tiles, then there are $(T_1,T'_1),(T_2,T'_2)\in\mathcal{BP}(\Phi)$ and non-parallel vectors $u_1,u_2$ so that $\{v\in\mathbb{S}^1:\langle v,u_i\rangle\,>0\}\subset S(T_i,T'_i)$ for $i=1,2$.

Let $T\in\Omega$ be fixed by $\Phi$. Let $\tau\in T$ and $x\ne0$ be such that $\tau-x\in T$. For $x_0\in\mathring{\tau}$ let $r>0$ be such that $T-x_0-y\sim_s T-x-x_0-y$ for all $y\in B_r(0)$ and  $T-x_0-y_0\nsim_sT-x-x_0-y_0$ for some $y_0$ with $|y_0|=r$. Then by Lemma \ref{branch points with direction} there is a branch pair, asymptotic in the positive half-plane determined by $z=z(\tau,x,x_0):=-y_0$. Hence we are done, unless all such $\tau,x,x_0$ produce parallel $z's$. For such $\tau,x,x_0,r$, and $y_0$, let $r_M=sup\{r':T-(x_0+(1-r'/r)y_0)-y\sim_s T-x-(x_0+(1-r'/r)y_0)-y$ for all $y\in B_{r'}(0)\}$.
It is easily checked that $r\le r'$. If $r'=\infty$, then $(T-x_0-y_0)-v\sim_s(T-x_0-y_0)-x-v$ for all $v$ with
$\langle v,-y_0\rangle\,>0$. By Lemma \ref{existence of R}, $(T-x_0-y_0)$ and $(T-x_0-y_0)-x$ would agree on
$\{v: \langle v,-y_0\rangle\ge R\}$ for some $R$. Pick $v_n$ with $\langle v_n,-y_0\rangle\to\infty$ and $(T-x_0-y_0)-v_n\to T'\in\Omega$. Then $(T-x_0-y_0)-x-v_n\to T'-x$ and $T'-x=T'$, contradicting non-periodicity of $\Omega$.
Thus $r\le r'<\infty$.

If it is not the case that $T-(x_0+(1-r_M/r)y_0)-y\sim_sT-x-(x_0+(1-r_M/r)y_0)-y$ for all $y\in cl(B_{r_M}(0))\setminus\{\pm (\frac{r_M}{r})y_0\}$, we are done, by Lemma \ref{branch points with direction}. Also, by maximality of $r_M$,  $T-(x_0+(1-r_M/r)y_0)-y\nsim_sT-x-(x_0+(1-r_M/r)y_0)-y$ for $y=\pm(\frac{r_M}{r})y_0$.

Let us suppose that all the $z(\tau,x,x_0)$ are parallel. Choose a unit vector $z$ parallel  to all these, and let $z^{\perp}$ be a unit vector perpendicular  to $z$. Let $S=S(\tau,x,x_0)$ be the infinite open strip $S:=\{sz^{\perp}+x_0+(1-t)y_0+t(\frac{r-2r_M}{r})y_0:s\in\R,0< t<1\}$. We claim that $T-y\sim_sT-x-y$ for all $y\in S$. For if this were not so, we could push parabolic curves out along $S$, as in the proof of Lemma \ref{corner}, until there occurs a first tangency with the set of those $y$'s such that $T-y\nsim_sT'-x-y$. The normal to the parabola at that point of tangency would not be parallel to $z$, leading to a $z(\tau',x,x'_0)$
not parallel with $z$. Furthermore, each boundary component of $cl(S)$ contains a point $y$ at which
$T-y\nsim_sT'-x-y$ (take $s=0$ and $t\in\{0,1\}$ in the definition of $S$). If not all points of $\partial S$ have this property, we are done by Lemma \ref{corner}. For any $x_0,x'_0\in\mathring{\tau}$, $S(\tau,x,x_0)$ and $S(\tau,x,x'_0)$ are either disjoint or equal; since tile interiors are connected, they must be equal. Hence $\mathring{\tau}\subset S(\tau,x,x_0)$ and the
boundary components of all the strips $S_{\tau,x}:=S(\tau,x,x_0)$ are $\epsilon$-characteristics for $T$
for any $\epsilon$ so that each tile contains an $\epsilon/2$-ball. Fix such an $\epsilon>0$.

For any fixed $x\ne 0$, the widths of the strips $S_{\tau,x}$ are uniformly bounded: otherwise, we could find $x_n$ so that $T-x_n-y \sim_sT-x-x_n-y$ for $y\in B_n(0)$, $T-x_n\to \bar{T}$, and $T-x-x_n\to\bar{T}-x$, and we would have $\bar{T}-y\sim_s\bar{T}-x-y$ for all $y$ so that $\bar{T}=\bar{T}-x$, by Lemma \ref{T=T'}, in contradiction to the non-periodicity of $\Omega$.
It follows that there must be an $r$ so that any strip of the form $S=\{sz+x_o-(1-t)rz^{\perp}+trz^{\perp}:s\in\R, 0<t<1\}$ contains a component of $\partial S_{\tau,x}$ for some $\tau$ and $x\ne 0$. Indeed, let
$P$ be a patch of $T$ that contains tiles $\tau$ and $\tau+x$ for some $x\ne0$. Then for each $w$
so that $P+w\subset T$, there is a strip $S_{\tau+w,x}$ that intersects $P+w$. Since the strips
$S_{\tau+w,x}$ have uniformly bounded widths and since there is an $R$ so that for every $y\in\R^2$ there is a $w$ with $spt(P+w)\subset B_R(y)$, we must have $r$ as desired.

Let $r$ be as above. It follows from Lemma \ref{finitely many cuts} (together with FLC) that $\{B_{2r}[T-y]:y\in \partial S_{\tau,x}$ for some $\tau\in T$ and $x\ne0$ with $\tau\in T-x\}$ is finite up to translation in the $z$ direction. Thus, for each boundary component $l$ of a strip $S_{\tau,x}$ there is a $y\in l$ and $w=w(l)\ne0$ parallel
with $z$ so that $B_{2r}[T-y]=B_{2r}[T-y+w]$. Furthermore, we can choose the $w(l)$ to have uniformly
bounded lengths (if $\{B_{2r}[T-y]:y\in \partial S_{\tau,x}$ for some $\tau\in T$ and $x\ne0$ with $\tau\in T-x\}$ has $m$ elements, up to translation in the $z$-direction, and the diameters of these patches are all less than d, then we can take $|w(l)|\le md$). It is then a consequence of finite local complexity that $\{w(l):l$ a component of $\partial S_{\tau,x},\tau\in T, \tau+x\in T,x\ne0\}$ is finite.

Now let $l_0$ be a boundary component of some $S_{\tau,x}$, let $y_0\in l_0$ be such that
$B_{2r}[T-y_0]=B_{2r}[T-y_0+w(l_0)]$ and let $P_0:=B_{2r}[T-y_0]+y_0$. Then $P_0$ and $P_0+w(l_0)$ are patches in $T$ and, as above for $\tau,x$, there is a strip $S'_0:=S_{P_0,w(l_0)}$ so that
$T-y\sim_sT-w(l_0)-y$ for all $y\in S'_0$. Also, $\{y:d(y,l_0)<2r\}\subset S'_0$. Now let $l_1$ be a boundary component of some $S_{\tau,x}$ with $y_1\in l_1$, $r<\langle y_1,z^{\perp}\rangle<2r$, and with
$B_{2r}[T-y_1]=B_{2r}[T-y_1+w(l_1)]$. Let  $P_1:=B_{2r}[T-y_1]+y_1$ and  let $S'_1:=S_{P_1,w(l_1)}$.
Continuing in this manner, we construct lines $l_n$ and strips $S'_n$, $n\in\Z$. Letting $S_n$ be the intersection of $S'_n$ and the strip between $l_{n-1}$ and $l_{n+1}$, and $w_n=w(l_n)$, the hypotheses of Lemma \ref{strips} are satisfied, but the conclusion of that lemma is contradicted, as $\{w_n\}$ finite.

Now suppose $\Omega$ has only branch line pairs and suppose that $(T,T')\in \mathcal{BP}(\Phi)$,
$v\in\partial S(T,T')$,  $t_0\in\R^+$, and $\epsilon>0$ are such that $T-t_0v\nsim_sT'-t_0v$ and $T-tv\sim_sT'-tv$ for $t_0<t<t_0+\epsilon$. There is a $k>0$ so that there are arbitrarily large $n$ such that
$(B_0[\Phi^n(T-t_0v)], B_0[\Phi^n(T-t_0v)])$ and $(B_0[\Phi^{n+k}(T-t_0v)], B_0[\Phi^{n+k}(T-t_0v)])$ are the same up to translation. As in previous arguments, this translation must be parallel to $v$ and, as soon as $\lambda^n\epsilon$ is greater than the diameter of any tile, this translation must be zero.
Fix such an $n$ and let $(\bar{T},\bar{T}'):=(\cup_{m\ge0}\Phi^{mk}(B_0[\Phi^n(T-t_0v)],\cup_{m\ge0}\Phi^{mk}(B_0[\Phi^n(T'-t_0v)])$. Then $(\bar{T},\bar{T'})$ is a periodic branch pair with $S(\bar{T},\bar{T}')\supset S(T,T')$. But $\bar{T}-tv\sim_s\bar{T}'-tv$ for $0<t<\epsilon$. It follows from the periodicity of
$(\bar{T},\bar{T}')$ that $\bar{T}-tv\sim_s\bar{T}'-tv$ for all $t>0$. Thus $v\in S(\bar{T},\bar{T}')$ so that $(\bar{T},\bar{T}')$ is an isolated or corner pair.
\hspace{\stretch{1}} $\square$
\\
\\
\noindent{\em Proof of Theorem \ref{structure of branch locus}}. Let us first suppose that the branch locus of $\Phi$ is non-collapsing. In this case $\mathcal{PBL}(\Phi)=\mathcal{BL}(\Phi)$ (Proposition \ref{PBL=BL}) and, if $(T,T')\in \mathcal{BP}(\Phi)$ and $v\in\tilde{\partial}S(T,T')$, then $(T-tv,T'-tv)\in\mathcal{BP}(\Phi)$ for all $t\ge 0$. Since $\mathcal{BP}(\Phi)$ is closed (Lemma \ref{branch pairs closed}), we see that $\mathcal{BL}(\Phi)=\{T:(T,T')\in\mathcal{BP}(\Phi)\}$ and $\mathcal{BL}(\Phi)$ is also closed. Fix  $v\in \tilde{\partial}S(T,T')$ and $u\in\mathbb{S}^1$ with $\{y\in\mathbb{S}^1:\langle y,u\rangle\,>0\}\subset S(T,T')$ for some $(T,T')\in\mathcal{BP}(\phi)$. Suppose that $\tau_1,\tau_2\in T$, $t_1,t_2\in\R^+$, and $x\in\R^2$ are such that $t_iv\in int(\tau_i)$ for $i=1,2$, $\tau_1+x=\tau_2$, and $B_0[T'-t_1v]+x=B_0[T'-t_2v]$. Since $B_0[T-t_iv]$ is stably related to $B_0[T'-t_iv]$ on $int(spt(B_0[T-t_iv])\cap spt(B_0[T'-t_iv]))\cap\{y\in\R^2:\langle y,u\rangle\,>0\}$, while  $B_0[T-t_iv]$ is  not stably related to $B_0[T'-t_iv]$ on $\{t_iv\}$ for $i=1,2$, it must be the case that $x$ is parallel with $v$. That is, the pairs $(\tau_1,l\cap\tau_1)$ and $(\tau_2,l\cap\tau_2)$ are the same up to translation, where $l=l_v:=\{tv:t\in\R\}$. Since there are only finitely many translationally inequivalent pairs of patches of the form $(B_0[\bar{T}],B_0[\bar{T}'])$ with $(\bar{T},\bar{T}')\in\mathcal{BP}({\Phi})$ (Lemma \ref{finitely many patches} - note that, in the above, $(T-t_iv,T'-t_iv)\in\mathcal{BP}({\Phi})$ from the non-collapsing assumption), and since the set $\{v:v\in\tilde{\partial} S(T,T'),(T,T')\in\mathcal{BP}({\Phi)}\}$ is finite (Theorem \ref{finitely many sectors}), the collection $\{[(\tau,\tau\cap l)]\}$ of translation equivalence classes of pairs $(\tau,\tau\cap l)$, where $0\in\tau\in\ T$ and $l=l_v$ for some $T$ with $(T,T')\in\mathcal{BP}({\Phi})$ and $v\in\tilde{\partial} S(T,T')$, is finite. By first collaring tiles if necessary, we may assume that $\Phi$ forces the border so that
the map $\hat{\pi}:\Omega_{\Phi}\to \hat{X}_{\Phi}:= \inv F_{\Phi}$ onto the inverse limit of the substitution map $F_{\Phi}$ on the Anderson-Putnam complex $X_{\Phi}$, is a homeomorphism (\cite{AP}). The 2-dimensional CW-complex $X_{\Phi}$ is the quotient of $\R^2$ formed from $T\in\Omega_{\Phi}$ by identifying $x_1$ and $x_2$ if there are $\tau_1$ and $\tau_2$ in $T$ with $x_i\in \tau_i$ and $\tau_2=\tau_1+(x_2-x_1)$. Equivalently, $X_{\Phi}:=\{[(\tau,x)]:x\in \tau,\tau\in\ T, T\in\Omega_{\Phi}\}/\sim$,
with $\sim$ the transitive closure of the relation defined by $[(\tau_1,x_1)]\sim[(\tau_2,x_2)]$ if there are $T\in\Omega_{\Phi}$,  $\bar{\tau}_1,\bar{\tau}_2\in T$, and $y_1,y_2$ so that $\bar{\tau}_i=\tau_i+y_i$, $i=1,2$, and $x_1+y_1=x_2+y_2$. With this description of $X_{\Phi}$, $\pi:\Omega_{\phi}\to X_{\Phi}$ is given by $\pi(T):=[[(\tau,0)]]$, where $0\in\tau\in T$ and $[[(*,*)]]$ denotes the $\sim$-equivalence class of the translational equivalence class of $(*,*)$. We see that $K:=\pi(\mathcal{BL}({\Phi}))$ is a 1-dimensional simplicial complex contained in the finite union of line segments $\{[[(\tau,x)]]:x\in\tau\cap l_v\}$, with $0\in\tau\in T$, $(T,T')\in\mathcal{BP}(\Phi)$, and $v\in \tilde{\partial} S(T,T')$. With $f:=F_{\Phi}\mid_{K}$, we have that $\hat{\pi}\mid_{\mathcal{BL}(\Phi)}$ maps $\mathcal{BL}(\Phi)$ is homeomorphically onto $\inv f$, conjugating $\Phi\mid_{\mathcal{BL}(\Phi)}$ with $\hat{f}$. The set $V$ of vertices of $K$ is defined as follows. Let $V_1:=\{\pi(T): (T,T')\in\mathcal{BP}(\Phi)$ is an isolated or corner pair$\}$: $V_1$ is finite by Theorem \ref{finitely many sectors}. Let $V_2:=\{[[(\tau,0)]]\in K:0$ is a vertex of $\tau\}$: $V_2$ is finite under the assumption that all tiles are polygons. Let $V_3:=\{[[(\tau,0)]]:0$ is in the interior of an edge $\eta$ of $\tau$, $\tau\in T, (T,T')\in\mathcal{BP}(\Phi)$, and there is $v\in\tilde{\partial} S(T,T')$ with $v$ not parallel with $\eta\}$: $V_3$ is the set of points where the finitely many line segments in $K$ meet an
edge of a polygonal face of $X_{\Phi}$ transversely, so $V_3$ is finite. Then $V:=V_1\cup V_2\cup V_3$
is finite and forward invariant under $f$. The conditions for $f:K\to K$ to be expanding Markov are easy to check.

Now assume that $\Phi$ is self-similar Pisot. We will argue that if $(T,T')\in\mathcal{PBP}(\Phi)$ and $v\in\tilde{\partial} S(T,T')$, then the ray $\{tv:t\ge0\}$ crosses the tiles of $T$ in only finitely many ways. The result then follows as above. Let $g:\Omega_{\Phi}\to\hat{\T}_A$ be geometric realization onto the maximal equi-continuous factor and let $\hat{F}_A:\hat{\T}_A\to\hat{\T}_A$ be the hyperbolic automorphism with $g\circ\Phi=\hat{F}_A\circ g$. (Here $A$ is the $2d\times 2d$ block diagonal matrix whose blocks are the companion matrix of the Pisot inflation factor $\lambda$, $d=deg(\lambda)$, $F_A$ is the associated hyperbolic toral endomorphism, and $\hat{F}_A$ is the shift on the inverse limit $\hat{\T}_A:=\inv F_A$.) The map $g$ also semiconjugates translation in $\Omega_{\Phi}$ with an $\R^2$ action on $\hat{\T}_A$ that consists of translation along the 2-dimensional unstable (under $\hat{F}_A$) leaves in $\hat{\T}_A$: we will use the notation $g(T)-x:=g(T-x)$ for this action. More explicitly in coordinates, the $\R^2$ action on $\hat{\T}_A$ is given by $(z_0,\ldots,z_n,\ldots)-x= (z_0-\tilde{x},\ldots,z_n-\lambda^{-n}\tilde{x},\ldots)$, where $x\mapsto \tilde{x}$ is an isomorphism of $\R^2$ with the unstable space $W^u_A$ of $F_A$. We will also use the key fact that $g$ is boundedly finite-to-one (\cite{BKS}).

Given $(T,T')\in\mathcal{PBP}(\Phi)$ and $v\in\tilde{\partial} S(T,T')$, we claim there is $t_0>0$ and $\bar{T}\in\Omega_{\Phi}$, with $\bar{T}$ $\Phi$-periodic, so that $T-t_0v\in W^s(\bar{T})$. Indeed, if there is $t_1>0$ so that $T-t_1v\sim_sT'-t_1v$, then let $t_0:=inf\{a: T-tv\sim_sT'-tv\, \forall t\in[a,t_1]\}$.
Then $t_0>0$ and $T-t_0v\nsim_sT'-t_0v$. Let $n\in\N$ be large and $k\in\N$ be such that $(B_0[\Phi^{n+k}(T-t_0v)], B_0[\Phi^{n+k}(T'-t_0v)])$ is the same as $(B_0[\Phi^{n}(T-t_0v)], B_0[\Phi^{n}(T'-t_0v)])$, up to translation. Such translation must be parallel with $v$ (as, for example, in the proof of Proposition \ref{periodic branch pairs exist}) and then must be 0, since $B_0[\Phi^{n+k}(T-t_0v)]$ and $B_0[\Phi^{n+k}(T'-t_0v)]$ as well as $B_0[\Phi^{n}(T-t_0v)]$ and $ B_0[\Phi^{n}(T'-t_0v)]$, are stably related on $\{tv:t\ge0\}$ intersected with the supports of the appropriate patches. That is, $(B_0[\Phi^{n+k}(T-t_0v)], B_0[\Phi^{n+k}(T'-t_0v)])=(B_0[\Phi^{n}(T-t_0v)], B_0[\Phi^{n}(T'-t_0v)])$ and $T-t_0v\in W^s(\bar{T})$, with $\bar{T}:=
\cup_{n\ge0}\Phi^{nk}(B_0[T-t_0v])$ $\Phi$-periodic of period $k$.
If, instead, $T-tv\nsim_sT'-tv$ for all $t\ge0$, there are only finitely many patches of the form $B_0[T-tv]$
for $t\ge0$, up to translation parallel to $v$ (as argued in the non-collapsing case above). Let $n$ be the
$\Phi$-period of $T$. There then are arbitrarily large $t_1$  and $s\in\R^+$, $m\in\N$ with $B_0[T-t_1v]+sv\subset\Phi^{mn}(B_0[T-t_1v])$. Let $t_0:=t_1+\frac{s}{\lambda^{mn}-1}$ and let $\bar{T}:=\cup_{l\ge0}\Phi^{lmn}(B_0[T-t_0v])$. Then $\bar{T}$ is $\Phi$-periodic and $T-t_0v\in W^s(\bar{T})$.

To simplify notation, assume that $T$ is fixed by $\Phi$. Let $t_0$ and $\bar{T}$ be as above with $\bar{T}$ periodic of period $k$ under $\Phi$. Let $L_v=L_v(T)$ denote the ray $L_v:=\{g(T-tv):t\ge0\}=\{g(T)-tv:t\ge0\}$ in $\hat{\T}_A$. Note that $g(T-\lambda^{nk}t_0v)=g(\Phi^{nk}(T-t_0v))=\hat{F}_A^{nk}(g(T-t_0v))\to g(\bar{T})$ as $n\to\infty$ so that $g(\bar{T})\in cl(L_v)$ and $L_v$ meets the $\epsilon$-stable manifold $W^s_{\epsilon}(g(\bar{T})):=\{z\in\hat{\T}:d(\hat{F}_A^n(z),\hat{F}_A^n(g(\bar{T})))\le\epsilon\,\forall n\ge0\}$ for each $\epsilon>0$. For each $n\in\N$, let $\pi_n:\hat{\T}\to \T^{2d}=\R^{2d}/\Z^{2d}$ be the projection of $\hat{T}_A=\inv F_A$ onto the $n$-th coordinate space. We will prove now that $\pi_0(cl(L_v))$ is a sub-torus of $\T^{2d}$ of dimension at most $d$. (In fact the dimension of $\pi_0
(cl(L_v))$ must then be exactly $d$, but we won't need that here.)

To see that this is the case, first note that the stable ($W^s_A$) and unstable ($W^u_A$) spaces of $A:\R^{2d}\to\R^{2d}$, at 0, have bases consisting of vectors with entries in the field $\Q(\lambda)$. ($W^u_A$ is spanned by a pair of eigenvectors for $A$ with eigenvalue $\lambda$, while $W^s_A$ is the space orthogonal to a space spanned by a pair of eigenvectors for $A^T$ with eigenvalue $\lambda$.)
Now, $g(T)-t_0v,g(T)-\lambda^kt_0v\in W^s(g(\bar{T}))$ means that the vector $w:=(\lambda^k-1)t_0\tilde{v}$ satisfies: (1)  $w-p\in W^s_A$ for some $p\in\Z^{2d}$. Also, since $\tilde{v}\in W^u_A$: (2) $w\in W^u_A$. Conditions (1) and (2) can be expressed as a system linear equations with coefficients in $\Q(\lambda)$; the unique solution $w=(w_1,\ldots,w_{2d})$ thus has all of its entries in $\Q(\lambda)$. The dimension of $\pi_0(cl(L_v))=cl(\{tw+\Z^{2d}:t\in\R\})$ equals the dimension over $\Q$ of the rational span of $\{w_1,\ldots,w_{2d}\}$. Since the dimension of $\Q(\lambda)$ over $\Q$ is $d$, the dimension of $\pi_0(cl(L_v))$ is at most $d$. The same argument clearly works for $\pi_n(cl(L_v))$ for all $n\ge0$, so $cl(L_v)$ is a $d'$-dimensional solenoid for some $d'\le d$.

We are going to show that, for small $\epsilon>0$, $W^s_{\epsilon}(g(\bar{T}))\cap cl(L_v)$ is a $d'-1$-dimensional cross-section to the $\R$-action $x\mapsto x-tv$ on the solenoid $cl(L_v)$, to which orbits must return with bounded gap. From this it will follow that there are $t'_i$ with $t'_{i+1}-t'_i$ bounded and with $t'_i\to\pm\infty$ as $i\to\pm\infty$ so that $\{B_0[T-t'_iv]\}$ is finite. An immediate consequence is that $\{tv:t\in\R\}$ intersects the tiles of $T$ in only finitely many ways.

There is $\delta=\delta(v)>0$ so that if $x+\Z^{2d},y+\Z^{2d}\in\pi_0(cl(L_v))$ and $|x-y|<\delta$, then $sx+(1-s)y+\Z^{2d}\in\pi_0(cl(L_v))$ for all $s\in\R$. We claim that for $\epsilon<\delta$, $B_{\epsilon/2}(\pi_0(g(\bar{T})))\cap\pi_0(cl(L_v))\subset\pi_0(W^s_\epsilon(g(\bar{T}))$. If this were not the case, then,
by the choice of $\delta$, $W^u_{\epsilon/2}(\pi_0(g(\bar{T})))\subset \pi_0(cl(L_v))$ and then $W^s(\pi_0(g(\bar{T})))\subset \pi_0(cl(L_v))$. But $W^s(\pi_0(g(\bar{T})))$ is dense in $\T^{2d}$ and  $\pi_0(cl(L_v))$ is a proper closed subset of $\T^{2d}$. Clearly, this $\delta$ works  for the corresponding statement in all coordinate projections $\pi_n$ simultaneously.

Let $\Delta_n(\epsilon)$ denote the local $\epsilon$ stable manifold at the point $\pi_n(g(\Phi^{-n}(\bar{T})))$ for the restriction of $F_A$ to the torus $\T:=\pi_0(cl(L_v))=\pi_n(cl(L_v))$. Then $\Delta_n(\epsilon/2)$ is contained in $W^s_{\epsilon}(\pi_n(g(\Phi^{-n}(\bar{T}))))$ for $\epsilon<\delta$ with $\delta$ depending only on $v$ (and not on $T$,  $\bar{T}$ or $n$). We saw above that the local unstable manifold $W^u_{\epsilon/2}(\pi_n(g(\Phi^{-n}(\bar{T}))))$ of $F_A$ meets $\T$ only along the arc $\pi_n(g(\Phi^{-n}(\bar{T})))-t\tilde{v}$, $|t|<\epsilon/2$, and it follows that $\Delta_n(\epsilon)$ is a co-dimension one disk in $\T$ transverse to the $\R$-action in the direction of $\tilde{v}$. Hence, orbits of the $\R$-action on $\T$ return to $\Delta_n(\epsilon/2)$ with bounded gap. (We really have a different $\R$-action on $\T=\pi_n(cl(L_v))$ for each $n$, namely $(y,t)\mapsto y-\lambda^{-n}t\tilde{v}$. We will take this into account.)

Now choose $\epsilon'$ small enough so that if $S,S'\in\Omega_{\Phi}$, $g^{-1}(\{g(S)\})=\{S_1,\ldots,S_m\}$, and $g(S')\in W^s_{\epsilon'}(g(S))$ then $B_0[S']=B_0[S_i]$ for some $i\in\{1,\ldots,m\}$. Let $n$ be large enough so that if $z,z'\in\hat{\T}_A$ are such that $\pi_n(z)=\pi_n(z')$ then $d(z,z')<\epsilon'/2$ and let $\epsilon>0$ be less that the $\delta$ above and also small enough so that if $d(\pi_n(z),\pi_n(z'))<\epsilon$ then $d(z,z')<\epsilon'$. There is then a $b$, depending only on $\epsilon$ and $v$, and $\{t_i':i\in\Z\}$ with $t'_i\to\pm\infty$ as $i\to\pm\infty$ and $0<t'_{i+1}-t'_i\le b$ for all $i\in\Z$, so that $\pi_n(g(T-t'_iv))=\pi_n(g(T))-t'_i\lambda^{-n}\tilde{v}\in\Delta_n(\epsilon/2)$ for all $i\in\Z$.
Let $g^{-1}(\{g(\bar{T})\})=\{\bar{T}_1,\ldots,\bar{T}_m\}$. We have that $B_0[T-t'_iv]\in\{B_0[\bar{T}_1],\ldots,B_0[\bar{T}_m]\}$ for all $i\in\Z$.

With $b$ as above, let $B$ be the maximum of the cardinalities of the collections $\{B_b[S']:S'\in\Omega_{\Phi}$ and $B_0[S']=B_0[S]\}$ over all $S\in\Omega_{\Phi}$. Finite local complexity insures that $B<\infty$. Let $M$ be the maximal cardinality of a fiber $g^{-1}(\{g(S)\})$ over all $S\in\Omega_{\Phi}$: $M<\infty$ by \cite{BKS}. The number of distinct translation equivalence classes of pairs
$(\tau,\tau\cap l)$ with $\tau\in T$, $l=\{tv:t\in\R\}$, and $l\cap\tau\ne\emptyset$, is bounded by $MB$ and
this bound depends only on $v$ and not on $T$ (provided, of course, that there is $T'$ with $(T,T')\in\mathcal{PBP}(\Phi)$ and $v\in\tilde{\partial} S(T,T')$).

We prove now that for a given $v\in\mathbb{S}^1$
there are only finitely many distinct sets $cl\{T-tv:t\in\R\}$ where $T$ varies over all members of periodic line pairs $(T,T')$ with $v\in\partial S(T,T')$. Indeed, there is a finite collection of pairs of patches
$\{(P_1,P_1'),\ldots,(P_n,P'_n)\}$ so that if $(T,T')$ is any line pair with $v\in \partial S(T,T')$ then $(B_0[T],B_0[T'])=(P_i-tv,P'_i-tv)$ for some $i\in\{1,\ldots,n\},t\in\R$. Suppose that $(T_1,T'_1)$ and $(T_2, T'_2)$ are periodic line pairs with $v\in\partial S(T_i,T'_i),i=1,2$ and $(B_0[T_1],B_0[T'_1])=(B_0[T_2]-sv,B_0[T'_2]-sv)$ for some $s\in\R$. Let $m$ be a common $\Phi$-period for $T_1$ and $T_2$.
Then $B_0[T_1]=B_0[T_2-sv]$ implies that $\Phi^{km}(T_2-sv)=T_2-\lambda^{km}sv\to T_1$ as $k\to\infty$. Thus, $T_1\in cl\{T_2-tv:t\in\R\}$. If $T_2-t_iv\to T_1$, then $T_2-(t_i+t)v\to T_1-tv$. Thus,
$cl\{T_1-tv:t\in\R\}\subset cl\{T_2-tv:t\in\R\}$. Interchanging 1 and 2 and replacing $s$ by $-s$ in the foregoing shows that $cl\{T_1-tv:t\in\R\}=cl\{T_2-tv:t\in\R\}$.

Note that if $(T,T')$ is a periodic line pair with $v\in\tilde{\partial} S(T,T')$ then $-v\in\tilde{\partial} S(T,T')$
so the image $\pi(cl(\{T-tv:t\in\R\}))$ in the Anderson-Putnam complex $X_{\Phi}$ equals $\cup_{w\in\tilde{\partial} S(T,T')} cl(\{T-tw:t\ge0\})$ and hence this latter set is a finite union of line segments in $X_{\Phi}$. As there are only finitely many sets $cl(\{T-tv:t\in\R\})$ for $(T,T')$ a periodic line pair and $v\in\tilde{\partial} S(T,T')$, there are only finitely many such $v$, and there are only finitely many corner and isolated pairs, $K:=\pi(\mathcal{PBP}(\Phi))$ is a finite union of line segments (some may be degenerate, if there are isolated pairs)
in $X_{\Phi}$. The description of $\mathcal{PBP}(\Phi)$ as $\inv (f:K\to K)$ proceeds as in the non-collapsing case.
\hspace{\stretch{1}} $\square$

\begin{lemma}\label{triple} Suppose that $\alpha_1$ and $\alpha_3$ are arcs on $\mathbb{S}^1$ with the property that $\alpha_1\cup\alpha_3$ contains a closed semi-circle in its interior. If there are distinct tilings $T_1,T_2,T_3\in\Omega_{\Phi}$ so that $T_i$ and $T_2$ are uniformly asymptotic in directions $\alpha_i$, $i=1,3$, then there is $w\in\R^2$ so that $T_2-w$ is $\Phi$-periodic.
\end{lemma}
\begin{proof} Let $T^n_i=\Phi^{-n}(T_i)$, $i=1,2,3$. By Lemmas \ref{asymptotic directions} and \ref{existence of R}, there is an $R$ so that $B_0[T^n_i-tv]=B_0[T^n_2-tv]$ for all $t\ge R$ and $v\in\alpha_i$, $i=1,3$. It follows that there are only finitely many triples of patches $(B_0[T^n_1],B_0[T^n_2], B_0[T^n_3])$, $n\ge0$, up to translation. We may then choose a subsequence and a $k\in\N$ so that the triples $(B_0[T^{n_j}_1],B_0[T^{n_j}_2], B_0[T^{n_j} _3])$ and $(B_0[T^{n_j+k}_1],B_0[T^{n_j+k}_2], B_0[T^{n_j+k}_3])$ are all translationally equivalent and (without loss of generality) all $n_i$ are divisible by $k$. By passing to a further subsequence, we may assume that $T^{n_j}_i\to\bar{T}_i\in\Omega_{\Phi}$ as $j\to\infty$ for $i=1,2,3$. Just as in the proof of Proposition \ref{branch points} (the ``$\alpha<0$" case) we may conclude that $\bar{T}_2$ is fixed by $\Phi^k$ and $\bar{T}_2=T_2-w$ for some $w$.
\end{proof}

\noindent{\em Proof of Theorem \ref{topological invariance 2}}. Let $\Omega=\Omega_{\Phi}$ and $\Omega'=\Omega_{\Psi}$. If $\Omega$ has an isolated branch pair or a corner branch pair, then by Lemma \ref{existence of R} such a pair is uniformly asymptotic in a set of directions containing a closed hemisphere in its interior and we're done by Theorem \ref{topological invariance 1}.

Let $h_0:\Omega_{\Phi}\to\Omega_{\Psi}$ be a homeomorphism for which there is a linear isomorphism $L$ of $\R^2$ so that $h_0(T-x)=h_0(T)-Lx$ for all $T\in\Omega_{\Phi}$ and all $x\in\R^2$ (Theorem 1.1 of \cite{Kwapisz}). Suppose that $(T_1,T_2)$ and $(T_2,T_3)$ are line branch pairs for $\Phi$ with
$\partial S(T_1,T_2)\ne\partial S(T_2,T_3)$. Then there is $w\in\R^2$ so that $T_2-w$ is $\Phi$-periodic
by Lemma \ref{triple} (in fact, it's not hard to show $w=0$ in this case). By the ``linearity" and uniform continuity of $h_0$, there are arcs $\alpha_i$ in $\mathbb{S}^1$, slightly smaller than the semi-circles
$\{Lv/|Lv|:v\in S(T_i,T_2)\}$, $i=1,3$ so that $\alpha_1\cup\alpha_3$ contains a closed semi-circle in its interior and so that $h_0(T_i)$ and $h_0(T_2)$ are uniformly asymptotic in directions $\alpha_i$, $i=1,3$. By Lemma \ref{triple}, there is $w'\in\R^2$ so that $h(T_2)-w'$ is $\Psi$-periodic. Let $h_1:\Omega_{\Phi}\to\Omega_{\Psi}$ by $h_1(T):=h_0(T+w)-w'$. Then $h_1$ is a homeomorphism that takes a $\Phi$-periodic orbit to a $\Psi$-periodic orbit and by Corollary 1.5 of \cite{Kwapisz}, there is a ``linear" homeomorphism $h:\Omega_{\Phi}\to\Omega_{\Psi}$ that conjugates some powers of $\Phi$ and $\Psi$. It follows easily that $h(\mathcal{BL}(\Phi))=\mathcal{BL}(\Psi)$ and $h(\mathcal{PBL}(\Phi))=\mathcal{PBL}(\Psi)$

Now suppose that $\Omega_{\Phi}$ is Pisot. We are done if $\Omega_{\Phi}$ has an isolated branch pair or a corner branch pair. Otherwise, there are are non-collapsing line branch pairs $(T_1,T'_1)$ and $(T_2,T'_2)$ with $\partial S(T_1,T'_1)\ne\partial S(T_2,T'_2)$ (Theorem \ref{two sectors}). Let $\partial
S(T_i,T'_i)=\{\pm v_i\}$, $i=1,2$, and let $g:\Omega_{\Phi}\to\hat{\T}_A$ be the map onto the maximal equicontinuous factor of the $\R^2$-action on $\Omega_{\Phi}$. Note that $g(T_i)=g(T'_i)$, $i=1,2$, since $T_i$ and $T'_i$ are proximal. Then $\cup_{t\in\R}W^s(g(T_i-tv_i))=
\cup_{t\in\R}W^s(g(T'_i-tv_i))$, $i=1,2$, are co-dimension one affine subspaces of $\hat{\T}_A$ that aren't
parallel (since $\{\pm v_1\}\ne\{\pm v_2\}$). Thus there are $t_1,t_2\in\R$ with $g(T_1-t_1v_1)=g(T'_1-t_1v_1)=g(T_2-t_2v_2)=g(T'_2-t_2v_2)$. It is proved in \cite{BKS} that, up to translation, there are only
finitely many sets of the form $\{B_0[T']:g(T')=g(T)\}$ for $T\in\Omega_{\Phi}$. Thus there are $n,k\in\N$
and $w\in\R^2$ so that $(B_0[\Phi^{n+k}(T_1-t_1v_1)],B_0[\Phi^{n+k}(T'_1-t_1v_1)],B_0[\Phi^{n+k}(T_2-t_2v_2)],B_0[\Phi^{n+k}(T'_2-t_2v_2)])=(B_0[\Phi^n(T_1-t_1v_1)]-w,B_0[\Phi^n(T'_1-t_1v_1)]-w,B_0[\Phi^n(T_2-t_2v_2)]-w,B_0[\Phi^n(T'_2-t_2v_2)-w])$. Equality of the first two coordinates of these quadruples implies that $w$ is parallel with $v_1$, while equality of the last two coordinates implies that
$w$ is parallel with $v_2$. Thus $w=0$ and the branch pairs $(\bar{T}_1,\bar{T}'_1)$ and $(\bar{T}_2,\bar{T}'_2)$ defined by $\bar{T}_i:=\cup_{m\ge0}\Phi^{mk}(B_0[\Phi^n(T_i-t_iv_i)]),\bar{T}'_i:=\cup_{m\ge0}\Phi^{mk}(B_0[\Phi^n(T'_i-t_iv_i)])$ are $\Phi$-periodic and $\partial S(\bar{T}_i,\bar{T}'_i)=\{\pm v_i\}$, $i=1,2$. Moreover, $g(\bar{T}_1)=g(\bar{T}'_1)=g(\bar{T}_2)=g(\bar{T}'_2)$ since $g$ is continuous,
$g\circ \hat{F}_A=\hat{F}_A\circ g$, and $\bar{T}_i=\lim_{m\to\infty}\Phi^{n+mk}(T_i-t_iv_i),\bar{T}'_i=\lim_{m\to\infty}\Phi^{n+mk}(T'_i-t_iv_i)$, $i=1,2$. Note that $\bar{T}_i\ne\bar{T}'_i$, $i=1,2$, by the non-collapsing condition satisfied by $(T_i,T'_i)$, $i=1,2$.

Now $\Omega_{\Psi}$ must also be Pisot. (For example, the linear nature of $h_0$ guarantees that if
$\phi$ is an eigenfuntion of the $\R^2$-action on $\Omega_{\Phi}$, then $\phi\circ h_0^{-1}$ is an eigenfunction of the $\R^2$-action on $\Omega_{\Psi}$, and these $\R^2$-actions have nontrivial eigenfunctions if and only if the inflations are Pisot, \cite{Sol}.) Also, the linearity of $h_0$ insures that $h_0$ preserves proximality. In particular, $g(S_1)=g(S'_1)=g(S_2)=g(S'_2)$ , where $g$ now denotes the map onto the maximal equi-continuous factor of the $\R^2$ action on $\Omega_{\Psi}$ and $S_i:=h_0(\bar{T}_i),S'_i:=h_0(\bar{T}'_i)$, $i=1,2$. Let $u_1\in S(\bar{T}_1,\bar{T}'_1)$ satisfy $\langle u_1,v_1\rangle=0$. Then for each $\delta>0$ there is $R=R(\delta)$ so that if $t>R$ then $d(\bar{T}_1-sv_1-tu_1,\bar{T}'_1-sv_1-tu_1)<\delta$ for all $s\in\R$ (Lemma \ref{existence of R}). Let $\epsilon>0$ be small enough so that if $S,S'\in\Omega_{\Psi}$ and $w\in\R^2$ are such that $d(S-tw,S'-tw)<\epsilon$ for all $t>0$ and $\lim_{t\to\infty} d(S-tw,S'-tw)=0$, then $B_0[S-tw]=B_0[S'-tw]$ for all $t>0$ (see the proof of Lemma \ref{asymptotic direction}). Now if $\delta>0$ is chosen so that $d(T,T')<\delta\Rightarrow d(h_0(T),h_0(T'))<\epsilon$, then, for $t>R(\delta)$, $B_0[S_1-sLv_1-tLu_1]=B_0[S'_1-sLv_1-tLu_1]$ for all $s\in\R$. Similarly for $S_2,S'_2$. Thus there is $w\in\R^2$ so that $S_i-w$ and $S'_i-w$ agree on
one of the components of $\R^2\setminus \{tLv_i:t\in\R\}$, $i=1,2$. For $n\in\N\cup \{0\}$, let $S_{n,i}:=\Psi^{-n}(S_i-w)$ and $S'_{n,i}:=\Psi^{-n}(S_i-w)$, $i=1,2$. Using the fact that there are only finitely many quadruples $(S_{n,1},S'_{n,1},S_{n,2},S'_{n,2})$, up to translation (for any $n$ the entries in the quadruple have the same image under $g$), and using that $Lv_1$ is not parallel with $Lv_2$ and $S_1\ne S'_1,S_2\ne S'_2$, we see,
exactly as in the proof of Proposition \ref{branch points}, that there is $w'\in\R^2$ so that $S_{0,i}-w'$ and $S'_{0,i}-w'$, $i=1,2$, are $\Psi$-periodic. Let $h_1:\Omega_{\Phi}\to\Omega_{\Psi}$ be given by $h_1(T):=h_0(T)-(w+w')$. Then $h_1$ is a homeomorphism that takes a $\Phi$-periodic orbit to a $\Psi$-periodic orbit. Corollary 1.5 of \cite{Kwapisz} supplies the desired $h:\Omega_{\Phi}\to\Omega_{\Psi}$, as above.
\hspace{\stretch{1}} $\square$

{\Small {\parindent=0pt Department of Mathematics, Montana State University,
Bozeman, MT 59717, USA \\
barge@math.montana.edu \\
\phantom{asdf} \\
Department of Mathematics, Southwest Minnesota State University, Marshall, MN 56258, USA \\
Carl.Olimb@smsu.edu
}}
\end{document}